\documentclass[10pt,english]{article}
\usepackage{geometry}
\usepackage{amsmath,amsfonts, amsthm,amssymb,oldgerm,amssymb}
\usepackage{array}
\usepackage[english]{babel}
\usepackage{float}
\usepackage{url}
\usepackage{cite} 
\usepackage{hyperref}
\usepackage{xcolor} 
\usepackage{comment}
\usepackage{enumitem}
\usepackage{amssymb,mathrsfs,graphicx,amsmath,graphicx,amsthm}
 \numberwithin{equation}{section}
\usepackage{mathtools}
\usepackage{array,multicol,comment,lipsum,caption}
\usepackage{color}

\newtheorem{theorem}{Theorem}[section]

\newtheorem{definition}{Definition}[section]

\newtheorem{lemma}{Lemma}[section]

\newtheorem{remark}{Remark}[section]

\newtheorem{proposition}{Proposition}[section]

\usepackage{epsfig}
\usepackage{relsize}
\usepackage{hyperref}
\usepackage[utf8]{inputenc}

\newcommand{\Int}{\displaystyle\int}

\newcommand{\R}{{\mathbb{R}}}
\newcommand{\N}{{\mathbb{N}}}

\newcommand{\oo}{{\mathcal{O}}}

\DeclarePairedDelimiter\abs{\lvert}{\rvert}
\DeclarePairedDelimiter\norm{\lVert}{\rVert}
\DeclarePairedDelimiterX{\inp}[2]{\langle}{\rangle}{#1, #2}

\makeatletter
\newcommand{\myitem}[1]{%
\item[#1]\protected@edef\@currentlabel{#1}%
}
\makeatother

\newcommand{\mycomment}[1]{}
\begin{document}

\title{\bf Analysis and stability of the FitzHugh-Nagumo system with rapidly oscillating Neumann boundary condition
}

\author{Alexander Rodr\'iguez\ \footnotemark[2] \hspace{.1cm }and Eduardo Cerpa\footnotemark[3]}


\footnotetext[2]{Instituto de Ingenier\'ia Matem\'atica y Computacional, Facultad de Matem\'aticas, Pontificia Universidad Católica de Chile, Avda.
Vicu\~na Mackenna 4860, Macul, Santiago, Chile. E.mail: eduardo.cerpa@uc.cl}

\footnotetext[2]{Facultad de Matem\'aticas, Pontificia Universidad Católica de Chile, Avda.
Vicu\~na Mackenna 4860, Macul, Santiago, Chile. E.mail: jarodriguez19@uc.cl}

\date{\today}

\maketitle

\begin{abstract}
In this paper we study the FitzHugh-Nagumo system posed on the half-line with a sinusoidal forcing term acting at the origin through the Neumann boun\-da\-ry condition.  By applying an averaging method, we show that, if the frequency of the oscillating forcing is high enough, then 
 the system can be approximated by a combination of a highly-oscillatory term and the solution of a simpler system. This approach allows us to establish stability results under persistent time-varying boundary conditions.  Along with the averaging method, we use an adaptation of the contracting rectangles method, tailored to handle parabolic equations with spatially and temporally varying coefficients.
\end{abstract}
\section{Introduction}
The FitzHugh-Nagumo (FHN) system \cite{fhn} is a mathematical model that describes the dynamics of membrane voltage in excitable systems. It has been instrumental in understanding neuron responses to various phenomena, including neural oscillations, cardiac pacing, and pattern formation in biophysical systems \cite{hodgkin1952quantitative,rattay1986analysis,MR1952568,MR2263523,MR381744}.
To gain a deeper insight into these processes, it is crucial to analyze the FHN system with boundary conditions that account for the rate of change of the membrane voltage. These conditions help in accurately modeling how the membrane voltage evolves and interacts with its environment, providing deeper insights into the dynamics of excitable systems and their behavior under various physiological scenarios. Mathematically, the FHN system is composed by two coupled parabolic equations and in this paper we consider it posed on the half-line with a forcing at the origin. Thus, the system reads as
\begin{equation}\label{ecuacionfhnneuman}
\left\{\begin{array}{ll}
f_t= f_{x x}+ f - \dfrac{f^3}{3}-g, & x \geqslant 0, t \geqslant 0, \\
g_t= dg_{x x} + \varepsilon(f-\gamma g + \beta), & x \geqslant 0, t \geqslant 0, \\
f(0, x)=f_0(x),\quad  g(0, x)=g_0(x), & x \geqslant 0, \\
f_x(t, 0)=h(t), & t \geqslant 0,
\end{array}\right.
\end{equation}
where $f=f(t,x)$ represents the membrane potential and $g=g(t,x)$ is a recovery variable, $h(t)$ is a source acting from the boundary, $d\geq 0$ is a diffusion coefficient, $f_0, g_0$ are the initial conditions and $\varepsilon$, $\gamma$, $\beta$ are positive constants. In the absence of boundary effects, numerous results have been established regarding the well-posedness of the system \cite{RAUCH197812,MR425397}, and the existence and stability of traveling wave solutions \cite{MR760971,MR3551276,MR3395129,MR3804135,MR1744043,MR1426757,MR4622876,MR4881333,MR430536}. In the case where boundary effects are present, various studies have considered Dirichlet, Neumann, or more general boundary conditions to understand their impact on the dynamics \cite{SCHONBEK1978119,MR4583896,MR3927583}.\\

Multiple and relevant phenomena on the behavior  of biological systems, occur with oscillating forcing. Understanding the influence of periodic external stimuli on excitable media is crucial for advancing our knowledge of dynamic biological systems. Incorporating trigonometric functions into Neumann boundary conditions provides a robust framework for exploring how rhythmic inputs affect system behavior, revealing insights into wave propagation, synchronization, and pattern formation \cite{MR2674516,MR952149,MR1952568}. Periodic boundary conditions model real-world scenarios where external oscillations, such as those from neural or cardiac stimulation, drive system dynamics. This approach not only deepens theoretical understanding but also has practical implications for designing therapeutic strategies and technologies.\\

High-frequency stimulation (HFS) of neurons represents a significant area of research at the intersection of mathematical modeling and medical applications. From a mathematical standpoint, the FitzHugh-Nagumo (FHN) equations serve as a fundamental framework for understanding how HFS influences neural dynamics \cite{PhysRevE.73.061102,PhysRevE.108.014207,MR4166598}. This mathematical approach is crucial for developing and refining models that predict neural responses and optimize stimulation parameters. Medically, HFS, exemplified by deep brain stimulation (DBS), is used to treat neurological disorders such as Parkinson's disease, epilepsy, and dystonia. This combination of mathematical modeling and clinical application offers a promising pathway to refine therapeutic approaches and improve patient outcomes.\\

Building upon the work presented in \cite{cerpa2023approximation, cerpaode}, this paper extends the investigation of reaction-diffusion system with boundary conditions. Our study further explores neuron responses under a different setting, aiming to enhance understanding and address the limitations identified in their work. We consider the one-dimensional FHN system with a trigonometric source, that is composed of a single sinusoidal expression with a small amplitude. These kinds of inputs elucidate how rhythmic external signals affect neural activity and synchronization \cite{MR2674516,MR2263523,cross1993pattern}. We study this setting using techniques of averaging as done in \cite{MR2316999,MR2881590,weinberg2013high,Khalil:1173048,MR1844527,MR2482394}, finding an approximate system for which we present some stability results. However, unlike the previous study, where explicit expressions allowed for the decomposition of the solution into fast and slow components, our work encounters additional complexities. Specifically, the absence of straightforward expressions for decomposing the solution necessitates a novel approach to apply the averaging method. This challenge requires us to develop and implement alternative strategies to approximate the solution and analyze the behavior of the system effectively. Beside employing the averaging method, another critical element of this work is the use of the contracting rectangles technique, which has been applied in previous studies \cite{SCHONBEK1978119,RAUCH197812,MR445091,MR1036209,MR3262177}.
\section{Setting}
Let us start this section by introducing all the functional spaces appearing in this paper.

\begin{definition} By denoting by $\R$ the set of real numbers and $\R^+$ the interval $[0,\infty)$, we set the following spaces involving functions $U: \R^+ \to \R$, for $k\geq 0$ an integer number:
$$
\begin{aligned}
&B C^k(\R^+)=\{U:(d / d x)^j U  \, \text{is a bounded uniformly continuous function on } \, \mathbb{R}^+ \, \text{for} \, \, 0 \leqslant j \leqslant k\}, \\
& C_0^k(\R^+)=\left\{U \in B C^k: U(0)=0 \text{ and }\lim (d / d x)^j U=0 \text { as }x \rightarrow \infty \text { for } 0 \leqslant j \leqslant k\right\},\\
&W^{k,p}(\R^+)=\left\{U \in L^p\left(\mathbb{R}^{+}\right):(d / d x)^j U \in L^p\left(\mathbb{R}^{+}\right) \text { for } 0 \leqslant j \leqslant k\right\}.
\end{aligned}
$$
\end{definition}

Now, we consider the system \eqref{ecuacionfhnneuman} with the boundary source $h(t)= A\sin(\omega t)$ with a small constant $A$ and a rather big frequency $\omega \gg 1$. In order to study this case, we first do a change of variable to find an equivalent problem, so we take a lift off function $\psi \in C^2_0(\R^+)$ such that $\psi^{\prime}(0)=-1$. Thus, considering  
\begin{equation}
    f = \overline{f} - h(t)\psi(x), \quad \text{and} \quad g = \overline{g},
\end{equation}
we got the system
\begin{equation}\label{ecuacionfhninternal}
\left\{\begin{aligned}
\overline{f}_t-\overline{f}_{x x}=& (\overline{f} - h\psi) - \dfrac{(\overline{f}-h\psi)^3}{3}-\overline{g} + \dot{h}\psi - h\psi^{\prime \prime}, & x \geqslant 0, t \geqslant 0, \\
\overline{g}_t- d\overline{g}_{x x} =& \varepsilon(\overline{f}-\gamma \overline{g} + \beta) + \varepsilon h\psi, & x \geqslant 0, t \geqslant 0, \\
\overline{f}(0, x)=f_0(x)&,\quad  \overline{g}(0, x)=g_0(x), & x \geqslant 0, \\
\overline{f}_x(t, 0)=0&, & t \geqslant 0.
\end{aligned}\right.
\end{equation}
Now for the initial data of system \eqref{ecuacionfhninternal} we impose that
\begin{equation*}
    f_0(x) \to v_0, \quad g_0(x) \to u_0 \quad \text{as} \quad x \to \infty
\end{equation*}
where $(v_0,u_0) \in \R^2$ is the unique solution of 
\begin{equation}\label{centerequation}
    \left\{\begin{array}{cc}
         0 =& v_0 - v_0^3/3 - u_0,  \\
         0 =& v_0 - \gamma u_0 +\beta.
    \end{array}\right.
\end{equation}
Consequently, the point $(v_0,u_0)$ is a stationary solution for \eqref{ecuacionfhninternal} when $\psi = 0$. To study the system around the stationary solution given by \eqref{centerequation},  we set $v = \overline{f} - v_0, u = \overline{g} - u_0, \overline{f}_0 = f_0 - v_0, \overline{g}_0 = g_0 - u_0$ to get
\begin{equation}\label{centeredfhn}
\left\{\begin{array}{rclcl}
v_t - v_{xx} &=& \dot{h}\psi - h\psi^{\prime \prime} + (v-h\psi) - v^3/3 - v^2v_0 - vv_0^2 + v^2h\psi  &&  \\
&&+ \,\, 2vv_0h\psi +v_0^2h\psi - vh^2\psi^2 - v_0h^2\psi^2 + \frac{1}{3} h^3\psi^3 - u, & &x\geq 0, t\geq 0,\\
u_t - d u_{xx} &=& \varepsilon(v - \gamma u) + \varepsilon h\psi, & & x\geq 0, t\geq 0, \\
v(0,x) &=& \overline{f}_0(x),  \quad u(0,x) = \overline{g}_0(x), & & x\geq 0, t\geq 0, \\
v_x(0,x) &=& 0 & & x \geq 0.
\end{array}\right.
\end{equation}
Regarding the parameters, we work 
in the same set as in \cite{cerpa2023approximation}. 

\begin{definition}
The parameters $\beta, \gamma$ are said to be admissible if 
\begin{equation}\label{parameters}
\frac{(1-\gamma)^3}{\gamma^3}+\frac{9}{4} \frac{\beta^2}{\gamma^2}>0 \text { and if } \exists \delta, 0<\delta<1 / 16 \text { such that }\left(1+\frac{1}{\delta \gamma}\right)^{1 / 2}\left(2-\frac{3+1 / \delta}{\gamma}\right)+3 \frac{\beta}{\gamma}>0 .
\end{equation}
\end{definition}

The following result, proven in \cite{cerpa2023approximation}, tells us that the previous definition makes sense.

\begin{lemma}\label{parameters-lemma}
(See \cite[Lemma 8]{cerpa2023approximation}).
The set of admissible parameters is non-empty and for any $0<\delta<1 / 16$ it contains the set
$$
\left\{(\beta, \gamma): \gamma \geq \frac{2 \delta+1}{\delta}, \beta \geq \frac{2}{3} \gamma\right\} .
$$

Additionally, if $\beta, \gamma$ are admissible then we have the following:
\begin{itemize}
    \item System \eqref{centerequation} has a unique solution $\left(v_0, u_0\right)$.

\item  For $\beta, \gamma$ admissible, let $\delta>0$ be such that \eqref{parameters} is fulfilled. Then the corresponding solution $\left(v_0, u_0\right)$ satisfies the bounds
$$
\min \{-\beta,-\sqrt{3}\} \leq v_0<-\sqrt{1+\frac{1}{\delta \gamma}}<0 .
$$

These bounds readily imply that
$$
\frac{1}{\max \left\{\beta^2-1,2\right\}} \leq \frac{1}{v_0^2-1}<\delta \gamma
.$$
\end{itemize}
\end{lemma}

Considering the functional spaces, let us recall the following norms and notations.

\begin{definition}
For any $f \in C\left([0, T] ; B C^0(\R^+)\right)$ we consider the norm
$$
\|f\|_Y=\sup _{0<t<T} \sup _{x \in \mathbb{R^+}}|f(t,x)| .
$$
Additionally, for a fixed $t$, we denote
$$
\|f(t,\cdot)\|_{L_x^{\infty}}=\sup _{x \in \mathbb{R^+}}|f(t,x)| .
$$
\end{definition}
\section{Main Results}
In order to study \eqref{centeredfhn}, we will use an averaging approach. The whole paper is about giving a rigorous framework to justify the approximations
$$v\approx V+w\quad \text{ and }\quad u\approx U,$$
where $(V,U)$ is the solution to what we call the Averaged System
\begin{equation}\label{partially-avg-system}
\left\{\begin{aligned}
V_t - V_{xx} =& V - V^3/3-V^2v_0 - Vv_0^2 - U ,  & & x\geq 0, t\geq 0, \\
U_t - d U_{xx} =& \varepsilon(V  - \gamma U) & & x\geq 0, t\geq 0, \\
V(0,x) = \overline{f}_0(x)&, \, U(0,x) = \overline{g}_0(x), & & x \geq 0,
\end{aligned}\right.
\end{equation}
and $w$ is the solution of the system
\begin{equation}\label{highfrequencyequation}
    \left\{\begin{array}{ll}
         w_t - w_{xx} = \dot{h}(t)\psi(x) + \left(1-(v_0-h(t)\psi(x))^2 \right)w, & x\geq 0, t\geq 0,\\ w(0,x) = 0, & x\geq 0,\\
         w_x(t,0) = 0, & t\geq 0,
    \end{array}\right.
\end{equation}
Our first main result establishes the stability of system \eqref{partially-avg-system}.
\begin{theorem}\label{stability-avg-system}
(Existence and stability for the Averaged System.) Let $\mathcal{B}$ denote either the space $\mathcal{B}=W^{k, p}(\mathbb{R^+})$ with $k, p \geq 1, k p>1$ or the space $\mathcal{B}=B C^0(\mathbb{R^+}) \cap L^p(\mathbb{R^+})$ with $p \geq 1$. Let $\varepsilon>0$ and let $\beta, \gamma$ be such that \eqref{parameters} holds. Let $\left(v_0, u_0\right) \in \mathbb{R}^2$ be the unique solution of \eqref{centerequation}. Given any open neighborhood $\mathcal{O}$ of $(0,0)$ there exists a contracting rectangle $R=[-L, L] \times[-S, S] \subset \mathcal{O}$ such that if the initial data $\left(f_0(x)-v_0, g_0-u_0\right) \in \mathcal{B} \times \mathcal{B}$ lies within R, then the system \eqref{partially-avg-system} with initial data $\left(f_0(x)-v_0, g_0(x)-u_0\right)$ has a unique solution $(V, U) \in$ $C([0, \infty) ; \mathcal{B} \times \mathcal{B})$ satisfying
$$
(V(t,x), U(t, x)) \in R, \quad \forall x \in \mathbb{R}^+, \quad t \geq 0 .
$$
\end{theorem}
The second main result establishes how the Averaged System \eqref{partially-avg-system} can be used to approximate the solution of system \eqref{ecuacionfhninternal}.
\begin{theorem}\label{theorem-approximation}
(Existence and approximation for the original system.) Let $\mathcal{B}$ denote the space $\mathcal{B}=W^{k, p}(\mathbb{R}^+)$ with $k \geq 3, p \geq 1$. Let $\varepsilon>0$ and $\beta$, $\gamma$ be such that \eqref{parameters} holds. Let $\left(v_0, u_0\right) \in \mathbb{R}^2$ be the unique solution of \eqref{centerequation}. Given any open neighborhood $\mathcal{O}$ of $(0,0)$ and any $\mu>0$, there exist $N>0$ and a contracting rectangle $R=[-L, L]\times[-S, S] \subset \mathcal{O}$ such that if the lift off function $\psi$ and $w$ solution of \eqref{highfrequencyequation} satisfy
$$
A \norm{\psi}_Y + A \norm{w}_Y\leq N
$$
and the initial data $\left(f_0(x)-v_0, g_0(x)-u_0\right) \in \mathcal{B} \times \mathcal{B}$ lies within $R$, then, for all $\omega$ large enough the solution $(f, g)$ of system \eqref{ecuacionfhninternal} with initial data $\left(f_0, g_0\right)$ exists in $C([0, \infty), \mathcal{B} \times \mathcal{B})$ and can be approximated using the solution $(V, U)$ of the  Averaged System \eqref{partially-avg-system} with initial data ($f_0- v_0, g_0-u_0$) in the following way
$$
\begin{array}{rll}
\left|f(t,x) + h(t)\psi(x)-\left(v_0+V(t,x)+w(t,x)\right)\right| \leq \mu, & \forall x \in \mathbb{R}^+, & t\geq 0, \\
\left|g(t,x)-\left(u_0+U(t,x)\right)\right| \leq \mu, & \forall x \in \mathbb{R}^+, & t\geq 0.
\end{array}
$$
\end{theorem}
The paper is organized as follows. In Section 4 we establish several well-posedness results that we use in the paper, starting with the proof of local existence and subsequently proving global existence. Section 5 deduces the Averaged System \eqref{partially-avg-system} and examines its stability properties. Finally, Section 6 presents the approximation results, starting with an analysis of the linear system followed by the nonlinear case.
\section{Well Posedness}
\subsection{Local Solvability}
Concerning the functional framework where our results will hold, we consider a family of spaces similar to the
one considered in \cite{RAUCH197812}.
\begin{definition}
A Banach space $\mathcal{B}$ of functions $w: \mathbb{R}^+\rightarrow \mathbb{R}$ is admissible if the following conditions hold:
\begin{enumerate}
    \myitem{A.1} $\mathcal{B}$ is a subset of $B C^0(\mathbb{R}^+)$ and for $w \in \mathcal{B},\|w\|_{\mathcal{B}} \geq\|w\|_{L^{\infty}}$. 
    \label{A.1}
    
\myitem{A.2} For every vector-valued function $F(U,x,t)$ we have the following property. For any $M>0$ there exist some constants $k_1, k_2>0$ such that for all $x \in \mathbb{R^+}$ and $t \in[0, \infty)$ we have
$$
\begin{gathered}
\|V\|_{\mathcal{B} \times \mathcal{B}} \leq M \text { and }\|W\|_{\mathcal{B} \times \mathcal{B}} \leq M \Rightarrow\|F(V, \cdot, t)-F(W, \cdot, t)\|_{\mathcal{B} \times \mathcal{B}} \leq k_1\|V-W\|_{\mathcal{B} \times \mathcal{B}}, \\
\|V\|_{\infty} \leq M \Rightarrow\|F(V, \cdot, t)-F(0, \cdot, t)\|_{\mathcal{B} \times \mathcal{B}} \leq k_2\|V\|_{\mathcal{B} \times \mathcal{B}}.
\end{gathered}
$$
If the function $F$ is independent of $x$ and $t$, then it also satisfies $F(0)=0$.\label{A.4}
\end{enumerate}
\end{definition}
Some examples of admissible spaces are $\mathcal{B}=W^{k, p}\left(\mathbb{R}^{+}\right)$ and $\mathcal{B}=B C^0(\mathbb{\R^+}) \cap L^p(\mathbb{\R^+})$ with $p \geq 1$.\\

We now consider the following non linear system of equations
\begin{equation}\label{general-equation-fhn}
    U_t = \tilde{A} U_{xx} + F(U) \quad x\geq 0, \, t\geq 0,
\end{equation}
where $U= (u_1,\cdots ,u_n)$ is a real vector, $F$ is a smooth $\R^n$ valued function with $F(0) =0$ and $\tilde{A} = diag\{a_1, \cdots , a_n\}$ is a diagonal matrix where $a_i  > 0$ for $1\leq i \leq p$ and $a_i=0$ for $p < i \leq n$. Let $G(t,x) = diag\{G_1(t,x),G_2(t,x), \cdots , G_n(t,x)\}$, where
$$
\tilde{G_i}(t,x) = \dfrac{1}{\sqrt{4a_i\pi t}}e^{\frac{-x^2}{4a_it}} \quad \text{if} \quad 1\leq i \leq p \quad \text{and} \quad G_i(t,x) = \delta (x) \quad \text{for} \quad p< i \leq n.
$$
For $h=\left(h_1, \ldots, h_n\right)  \in B C^0(\R^+)$ and $g=\left(g_1, \ldots, g_n\right) \subset \mathcal{B},$ let 
\begin{align}
G_i(t, z, x) & =\tilde{G_i}(t, z-x)+\tilde{G_i}(t,-z-x), \quad 1 \leqslant i \leqslant n , \\
H_i(t, x) & =- \Int_0^t h_i(s)  G_i(t-s, x)d s, \quad 1 \leqslant i \leqslant p , \label{integral-solution-for-H}\\
H_i(t, x) & =0, \quad p<i \leqslant n , \label{integral-solution-for-H-with-zero}\\
S_i(t, x) & =\int_0^{\infty} g_i(z) \tilde{G}_i(t, z, x) d z, \quad 1 \leqslant i \leqslant p , \label{integral-solution-for-S}\\
S_i(t, x) & =g_i(x), \quad p<i \leqslant n, \\
S_i(0, x) & =g_i(x), \quad 1 \leqslant i \leqslant n .
\end{align}

We recall that $H_{i}(t, x), 1 \leqslant i \leqslant p$, is the solution of 
\begin{equation}
\left\{\begin{aligned}
\partial_t H_i  &=&a_i\partial_x^2 H_i,  & & x\geq 0, t\geq 0, \\
H_i(0,x) &=& 0,  & & x \geq 0,\\
\partial_x H_i(t,0) &=& h_i(t), & &  t\geq 0,
\end{aligned}\right.
\end{equation}
and $S_i(t, x), 1 \leqslant i \leqslant p$, is the restriction to $x>0$ of the solution to the Cauchy problem
\begin{equation}
\left\{\begin{aligned}
\partial_t S_i  &=&a_i\partial_x^2 S_i,  & & x\in \R, t\geq 0, \\
S_i(0,x) &=& g_i(x)& & x \geq 0,\\
S_i(0,x) &=& g_i(-x)& & x < 0.\\
\end{aligned}\right.
\end{equation}
Let $R(t, x)$ be the unique element of $C([0, \infty) ; \mathcal{B} )$ such that
\begin{equation}
\left\{\begin{aligned}
R_t&=&\tilde{A}R_{x x}, && t>0, x>0, \\
R(0, x)&=&g(x), && x>0, \\
\frac{d}{d x}\left(R_1, \ldots, R_p\right)(t, 0)&=&h(t), && t>0 .
\end{aligned}\right.
\end{equation}

Now, due to Duhamel Principle we can write a solution $U \in C([0, \infty) ; \mathcal{B} \times \mathcal{B} )$ that satisfies \eqref{general-equation-fhn} with initial and boundary data
\begin{align}
U(0, x) & =g(x), \label{initial-condition-for-general-equation}\\
(U_i)_x(t, 0) & =h_i(t), \label{boundary-conditions-for-general-equation}
\end{align}

in the following form
\begin{equation}\label{integral-solution}
U(t, x)=R(t, x)+\int_0^t \int_0^\infty G(t-s, z, x) F(U(s, z)) d z d s . 
\end{equation}

\begin{remark}
$R(t,x)$ has the following form
\begin{equation}\label{equality-for-R}
    R(t,x) = H(t,x) + S(t,x),
\end{equation}
and also using \eqref{integral-solution-for-H}, \eqref{integral-solution-for-H-with-zero}, \eqref{integral-solution-for-S} we see that $R$ satisfies the estimate 
$$\| R \|_{C\left( [0,t_0);\mathcal{B}\right)}\leq \left(m\|h\|_{\infty}+\|g\|_{\infty}\right),$$
where $m^2=\dfrac{1}{4\pi \tilde{a}}$ with $\tilde{a}= \min \left(a_1, \ldots, a_p\right)$.
\end{remark}

We show the existence of the solution for a short time interval, the length of which depends only on $F$, and the norms of the initial and boundary data. This is stated in the following theorem. 
\begin{theorem}\label{theorem-local-solution}
For any $h \in BC^0(\R^+)$ and $g \in C_0(\R^+)$ there is a constant $t_0\in [0,1]$, depending only on $F,\|g\|_{\infty}$ and $\|h\|_{\infty}$, such that the Neumann problem \eqref{general-equation-fhn} with initial data \eqref{initial-condition-for-general-equation} and boundary data \eqref{boundary-conditions-for-general-equation} has a unique solution $U \in C\left(\left[0, t_0\right];\mathcal{B}\times \mathcal{B} \right)$ and
$$
\|U\|_{C\left(\left[0, t_0\right];\mathcal{B}\times \mathcal{B}\right)} \leqslant 2\left(m\|h\|_{\infty}+\|g\|_{\infty}\right).
$$
\end{theorem}
\begin{proof}
We first show that there is a $t_0>0$ depending only on $\left\|g\right\|_{\infty}$, $\|h\|_{\infty}$ and $F$ such that \eqref{general-equation-fhn} has a unique solution in $C([0, t]; \mathcal{B}\times \mathcal{B} )$ and $\|U\|_{C\left([0, t] \mathcal{B}\times \mathcal{B}\right)} \leqslant 2\left(m\|h\|_{\infty}+\|g\|_{\infty}\right) $. For any $t_0>0$, let
$$
\Omega=\left\{U \in C\left(\left[0, t_0\right];\mathcal{B}\times \mathcal{B}\right):\left\|U(t)-R(t,\cdot) \right\|_{\mathcal{B}\times \mathcal{B}} \leqslant\left\|g\right\|_{\infty}+m\|h\|_{\infty}, 0 \leqslant t \leqslant t_0\right\} .
$$

If $U \in \Omega$, then $\|U(t)\|_{\mathcal{B}\times \mathcal{B}} \leqslant 2\left(m\|h\|_{\infty}+\|g\|_{\infty}\right)$, for $0 \leqslant t \leqslant t_0$, so by \ref{A.4}, we may choose a constant $k$, independent of $t_0$ such that for $U, V \in \Omega$,

\begin{equation}\label{lipschitz-property-middleproof-localsolvae}
\|F(U(t))-F(V(t))\|_{C\left(\left[0, t_0\right];\mathcal{B}\times \mathcal{B}\right)} \leqslant k\|U(t)-V(t)\|_{C\left(\left[0, t_0\right];\mathcal{B}\times \mathcal{B}\right)} .
\end{equation}

Let $t_0=\dfrac{-\frac{4k}{\sqrt{\tilde{a}\pi}}+
\sqrt{\frac{16k^2}{\tilde{a}\pi}+8k}}{4k}$ with $\tilde a=\min \left(a_1, \ldots, a_p\right)$. Thus, this $t_0$ clearly depends only on $F$ and $\left\|g\right\|_\infty $. We define a map $\Gamma$ from $C\left(\left[0, t_0\right];\mathcal{B}\times \mathcal{B}\right)$ into itself by
$$
\Gamma U(t)=R(t,x)+\Int_0^t G(t-s) * F(U(s)) d s .
$$

We first show that $\Gamma$ maps the closed set $\Omega$ into itself. Using \eqref{lipschitz-property-middleproof-localsolvae} with $V=0$, we have, for $U \in \Omega$
$$
\left\|\Gamma(U)(t)-R(t,x)\right\|_{\mathcal{B}\times \mathcal{B}} \leqslant \dfrac{k}{\sqrt{\tilde{a}\pi}} \int_0^t \dfrac{\|U(s)\|_{\mathcal{B}\times \mathcal{B}}}{(t-s)^{1/2}} d s + k \int_0^t\|U(s)\|_{\mathcal{B}\times \mathcal{B}} ds.
$$
Hence, we obtain for $0 \leq t \leq t_0$,
$$
\left\|\Gamma(U)(t)-R(t,x)\right\|_{\mathcal{B}\times \mathcal{B}} \leqslant \left(\dfrac{4t_0^{1/2}k}{\sqrt{\tilde{a}\pi}} +2kt_0 \right) \left(\left\|g\right\|_{\infty}+m\|h\|_{\infty}\right)=\left\|g\right\|_{\infty} +m\|h\|_{\infty},
$$
so that $\Gamma$ maps $\Omega$ into itself. Next, let us show that $\Gamma$ is a contraction mapping on $\Omega$. If $U, V \in \Omega$, then
$$
\begin{aligned}
\|\Gamma(U)(t)-\Gamma(V)(t)\|_{\mathcal{B}\times \mathcal{B}} & \leqslant \int_0^t\|G(t-s) *(F(U(s))-F(V(s)))\|_{\mathcal{B}\times \mathcal{B}} d s \\
& \leqslant \dfrac{k}{\sqrt{\tilde{a}\pi}} \int_0^t\dfrac{\|U(s)-V(s)\|_{\mathcal{B}\times \mathcal{B}}}{(t-s)^{1/2}} d s  + k \int_0^t\|U(s)\|_{\mathcal{B}\times \mathcal{B}} ds\\
& \leqslant \left(\dfrac{2t_0^{1/2}k}{\sqrt{\tilde{a}\pi}} +kt_0 \right) \|U-V\|_{C\left(\left[0, t_0\right];\mathcal{B}\times \mathcal{B}\right)} \\
& \leqslant \frac{1}{2}\|U-V\|_{C\left(\left[0, t_0\right]; \mathcal{B}\times \mathcal{B}\right)}
\end{aligned}
$$
Taking the supremum on $t \in [0, t_0]$ we obtain that $\Gamma$ is a contraction and therefore it has a unique fixed point in $\Omega$. However, this still leaves open the possibility of finding solutions outside of $\Omega$. We cover uniqueness of the solution in Proposition \ref{proposition-for-uniqueness}.
To complete the proof we must extend the solution to an interval $0 \leqslant t \leqslant t_1$ which depends only on $\left\|g\right\|_{\infty}$. Since $\mathcal{B} \subset B C^0(\R^+)$, the above argument shows that there is a $t_1$ depending only on $F$ and $\left\|g\right\|_{\infty}$ and a solution $V \in C\left(\left[0, t_1\right] ;B C^0(\R^+)\times B C^0(\R^+)\right)$ with $\|V(t)\|_{\infty} \leqslant 2\left\|g\right\|_{\infty} + 2m\left\|h\right\|_{\infty}$ for $t \in\left[0, t_1\right]$. By uniqueness in $C\left(\left[0, t_0\right] ; B C^0(\R^+)\times B C^0(\R^+)\right)$ we have $U=V$ for $0 \leqslant t \leqslant t_0$. To complete the proof it suffices to show that $V \in$ $C\left(\left[0, t_1\right];\mathcal{B}\times \mathcal{B}\right)$; then $V$ provides the desired extension. To prove the regularity of $V$ we will show that there is an $\eta>0$ independent of $t_2 \in\left[0, t_1\right]$ with the property that if $V \in C\left(\left[0, t_2\right]\mathcal{B}\times \mathcal{B}\right)$ then $V \in C\left(\left[0, t_2+\eta\right];\mathcal{B}\times \mathcal{B}\right)$. A finite number of applications of this result implies that $V \in C\left(\left[0, t_1\right];\mathcal{B}\times \mathcal{B}\right)$. The main point is an estimate for $\|V\|_{C\left(\left[0, t_2\right];\mathcal{B}\times \mathcal{B}\right)}$ which is independent of $t_2$. We know that $V$ has the following form
\begin{equation}\label{middleproof-regularitylocal}
V(t, x)=R(t, x)+\int_0^t \int_0^\infty G(t-s, z, x) F(V(s, z)) d z d s.
\end{equation}
By using \ref{A.4} and the fact that $\|V(t)\|_{\infty} \leq 2\left\|g\right\|_{\infty}+2 m\|h\|_{\infty}$, we see that there is a $k_2$ so that
$$
\|F(V(s))\|_{\mathcal{B} \times \mathcal{B}} \leq k_2\|V(s)\|_{\mathcal{B} \times \mathcal{B}}.
$$
Taking norms both sides of \eqref{middleproof-regularitylocal} and applying Gronwall's inequality to
$$
\|V(t)\|_{\mathcal{B}\times \mathcal{B}} \leq \dfrac{1}{t^{1/2}}\|U^0\|_{\mathcal{B}\times \mathcal{B}}+k_2\int_{0}^t\left( \dfrac{1}{\sqrt{\tilde{a}\pi(t-s)}} + 1\right)\|V(s)\|_{\mathcal{B}\times \mathcal{B}} d s
$$
yield a constant $C>0$, independent of $t_2 \in\left[0, t_1\right]$, such that $\|V(t)\|_{\mathcal{B} \times \mathcal{B} } \leq C$ for $0 \leq t \leq t_2$. In particular $\left\|V(t_2)\right\|_{\mathcal{B}\times \mathcal{B} } \leq C$. From the proof of existence for $U\in [0,t_0]$  there exists $\eta>0$ only dependent on $C$ and $F$ so that the system for \eqref{general-equation-fhn} with data of $\mathcal{B} $ norm at most $C$ has a solution in $C([0, \eta] ; \mathcal{B}\times \mathcal{B}  )$. Let $W \in C([0, \eta] ; \mathcal{B}\times \mathcal{B} )$ solve \eqref{general-equation-fhn} with $W(0)=V\left(t_2\right)$ and then define $V\left(t_2+s\right)=W(s)$ for $0 \leqslant s \leqslant \eta$.  Since $\eta>0$ independent of $t_2$ allows us to continue this process until reaching $t_1$ and this completes the proof.
\end{proof}
Using the integral equation \eqref{integral-solution} we are able to show dependence on initial conditions and uniqueness of the solution in the following proposition.
\begin{proposition}\label{proposition-for-uniqueness}
(Uniqueness). The solution of \eqref{general-equation-fhn} is unique in $C\left(\left[0, T\right] ; \mathcal{B}\times \mathcal{B}\right)$ and depends con\-ti\-nuous\-ly on the initial conditions.
\end{proposition}
\begin{proof}
    Let $U$ and $\tilde{U}$ be solutions belonging to $C([0, T]; \mathcal{B} \times \mathcal{B})$, with initial conditions $U^0$ and $\tilde{U}^0$, respectively. Then, using \eqref{integral-solution-for-H},\eqref{integral-solution-for-H-with-zero}, \eqref{integral-solution-for-S}, \eqref{equality-for-R} and \eqref{integral-solution} we got for any $t \in[0, T]$
\begin{align*}
U(t)-\tilde{U}(t) &= H(t,x) + G(t) *U^0(0) +\int_0^t G(t-s) *F(U(s)) d s\\
&- H(t,x) - G(t) *\tilde{U}^0(0) +\int_0^t G(t-s) *F(\tilde{U}(s)) d s\\
&= G(t) *\left(U^0(0)-\tilde{U}^0(0)\right)+\int_0^t G(t-s) *(F(U(s))-F(\tilde{U}(s))) d s .  
\end{align*}
Suppose that $\|U(t)\|_{\mathcal{B}\times \mathcal{B}} \leqslant M$ and $\|\tilde{U}(t)\|_{\mathcal{B}\times \mathcal{B}} \leqslant M$ for $0 \leqslant t \leqslant T$, and choose $k$ so that \ref{A.4} holds. Then defining $\tilde{\alpha}(s) = \dfrac{1}{\sqrt{\tilde{a}\pi s}} + 1$, we get
$$
\|U(t)-\tilde{U}(t)\|_{\mathcal{B}\times \mathcal{B}} \leqslant \dfrac{1}{t^{1/2}}\|U^0(0)-\tilde{U}^0(0)\|_{\mathcal{B}\times \mathcal{B}}+ k\int_0^t \tilde{\alpha}(t-s)\|U(s)-\tilde{U}(s)\|_{\mathcal{B}\times \mathcal{B}} d s,
$$
so that Gronwall's inequality yields
$$
\|U(t)-\tilde{U}(t)\|_{\mathcal{B}\times \mathcal{B}} \leqslant \dfrac{e^{k\int_0^t\tilde{\alpha}(t-s)}}{t^{1/2}} \| U^0(0)-\tilde{U}^0(0) \|_{\mathcal{B}\times \mathcal{B}} .
$$
\end{proof}
\begin{remark}\label{remark-for-global-linearequation}
If the system \eqref{general-equation-fhn} is linear, i.e., $F(U)=FU$ for a constant matrix $F$, then for any $g=\left(g_1, \ldots, g_n\right) \in C_0(R^+)$, and $h=\left(h_1, \ldots, h_p\right) \in B C(R^+)$ there is a unique solution $U \in C([0, \infty) ; \mathcal{B} \times \mathcal{B})$ of \eqref{general-equation-fhn} with initial and boundary data \eqref{initial-condition-for-general-equation} and \eqref{boundary-conditions-for-general-equation}, respectively. Furthermore, there are constants $k$ and $c$, independent of $g$ and $h$, such that
$$
\|U(t)\|_{\infty} \leqslant k e^{c t}\left(\|g\|_{\infty}+\frac{\|h\|_{\infty}}{2 \pi^{1 / 2}}\right) .
$$
\end{remark}

\subsection{Global Solvability}
We start this subsection by defining a relevant notion introduced in \cite{RAUCH197812}.
\begin{definition}
(Contracting Rectangle). For $L, S>0$ let $R_{L, S}$ be the rectangle centered at $(0,0)$ defined as
$$
R_{L, S}=[-L, L] \times[-S, S] .
$$

Given a vector-valued function $H: \mathbb{R}^2 \times \mathbb{R} \times[0, T] \rightarrow \mathbb{R}^2$ we say that the rectangle $R_{L, S}$ is contracting under the vector fields $H$ if for each $\vec{v} \in \partial R_{L, S}$ and every outward normal unit vector $n(\vec{v})$ we have
\begin{equation}\label{normal}
    \sup _{x \in \mathbb{R}} n(\vec{v}) \cdot H(\vec{v}, x, t)<0, \quad t \in[0, T].
\end{equation}

\end{definition}
\begin{remark}
    If $\vec{v}$ is in the corner of the rectangle, we assume \eqref{normal} is satisfied for all $n(\vec{v})$ in the closed cone outward normal to the boundary.
\end{remark}

Given this definition we can show the existence of large contracting rectangles that help us to prove the global solvability of the solution of \eqref{ecuacionfhnneuman}. This is stated in the following Lemma.

\begin{lemma}\label{lemma-of-large-rectangle}
    Let $\phi: \mathbb{R}^2 \rightarrow \R$ be such that there exists $M>0$ that satisfies $\sup _{x_i \geqslant 0}|\phi(x, t)| \leq \tilde{M}$. Then, there exists a contracting rectangle $R=R_{L,S}$ for the vector field
    $$
F_\phi(f,g)=\left(f+\phi - \frac{(f+\phi)^3}{3} - g, \varepsilon (f - \gamma g)\right) .
$$
\end{lemma}
\begin{proof}
    See Appendix \ref{appendix-proof-lemma-largerectangle}.
\end{proof}

As we mentioned above, Lemma \ref{lemma-of-large-rectangle} leads to the following existence theorem.

\begin{theorem}\label{theorem-global-solution}
Let $f_0, g_0 \in C_0(R^+)$. If $h \in B C(R^+)$, then there exists a unique solution $H=(f, g) \in C\left([0, \infty);\mathcal{B}\times \mathcal{B}\right)$ to the Neumann problem \eqref{ecuacionfhnneuman}. Furthermore, for any $T \geqslant 0$,
$$
\|H(t)\|_{\infty} \leqslant \theta\left(T,\|h\|_{\infty},\left\|f_0\right\|_{\infty},\left\|g_0\right\|_{\infty}\right), \quad 0 \leqslant t \leqslant T,
$$
where $\theta$ grows at most exponentially in $T$.
\end{theorem}
\begin{proof}
To achieve global existence, we only need to establish an a priori bound for the solution. It is sufficient to demonstrate that there is an a priori bound on the interval $[0, T]$, where $T > 0$ is arbitrary. To do this, we construct a comparison function $\Phi(t, x)=(\phi(t, x), \alpha(t, x))$, bounded on $[0, T]$, which has the same initial and boundary values as the solution of \eqref{ecuacionfhnneuman}. Let $\Phi(t, x)=(\phi(t, x), \alpha(t, x))$ be the solution of
$$
\begin{aligned}
& \phi_t=\phi_{x x}-\alpha, \\
& \alpha_t=d\alpha_{xx} + \varepsilon(\phi-\gamma \alpha + \beta),
\end{aligned}
$$
with initial and boundary conditions
$$
\begin{gathered}
\phi(0, x)=f_0(x), \quad \alpha(0, x)=g_0(x), \\
\phi_x(t, 0)=h(t) .
\end{gathered}
$$
By Remark \ref{remark-for-global-linearequation} we know that $\Phi$ has at most exponential growth. Thus, there exists $N>0$ such that $\|\Phi(t)\| \leqslant N$ for $0 \leqslant t \leqslant T$. Now, we are going to estimate $\hat{H}=H-\Phi$.
By setting
$$
\begin{aligned}
& \tilde{f}=f-\phi, \\
& \tilde{g}=g-\alpha,
\end{aligned}
$$

we get that
\begin{equation}
\left\{\begin{array}{ll}
\tilde{f}_t=\tilde{f}_{x x}+\tilde{f}+\phi - (\tilde{f}+\phi)^3/3-\tilde{g}, & x \geqslant 0, t \geqslant 0, \\
\tilde{g}_t=d\tilde{g}_{xx}+\varepsilon( \tilde{f}-\gamma \tilde{g}), & x \geqslant 0, t \geqslant 0, \\
\tilde{f}(0, x)=0, \tilde{g}(0, x)=0, & x \geqslant 0, \\
f_x(t, 0)=0, & t \geqslant 0.
\end{array}\right.
\end{equation}

By Lemma \eqref{lemma-of-large-rectangle} we can construct a rectangle, $R(T)$ with $0 \in \operatorname{int} R(T)$, 
which 
is contracting for the vector field $F_\phi(\tilde{H})=\left(\tilde{f}+\phi - \frac{(\tilde{f}+\phi)^3}{3} - \tilde{g}, \varepsilon (\tilde{f} - \gamma \tilde{g})\right)$.\\

We aim to demonstrate that $(\tilde{f}, \tilde{g})$ remains bounded for $0 \leq t \leq T$. Moreover, given that $(\tilde{f}(0, x), \tilde{g}(0, x))$ lies within $R(T)$, we will establish that $(\tilde{f}(t, x), \tilde{g}(t, x))$ stays in $R(T)$ for $0 \leq t \leq T$. To do this, we will use a proof by contradiction to show that $(\tilde{f}, \tilde{g})$ cannot reach the boundary of $R(T)$. Let us assume that $(\tilde{f}, \tilde{g})$ reaches the boundary of $R(T)$. Given that the initial condition is in $C_0$, there must be a first time $t_0$ and a finite $x_0$ such that $\tilde{H}(t_0, x_0) \in \partial R$. Note that $x_0$ cannot be zero, since $0$ is within the interior of $R(T)$. If we are on the right-hand side of $R(T)$, then $\tilde{f}(t_0, x_0)$ must lie on the boundary $\partial R(T)$. Since $t_0$ is the first time, we have
\begin{equation}\label{equation-for-contradicition-middleproof-globalsolution}
    \tilde{f}_t\left(t_0, x_0\right) \geqslant 0 .
\end{equation}

By construction of $R(T)$ and definition of contracting rectangle we have
$$
(\tilde{f}+\phi - \frac{(\tilde{f}+\phi)^3}{3} - \tilde{g})\left.\right|_{(t, x)=\left(t_0, x_0\right)}<0 .
$$

Since $\tilde{f}(t, x) \leqslant \tilde{f}\left(t_0, x_0\right)$ for all $x \geqslant 0, t \leqslant t_0$, we see that $\tilde{f}\left(t_0, \cdot\right)$ has a local maximum at $x_0$, so $\tilde{f}_{x x}\left(t_0, x_0\right) \leqslant 0$. Thus at $\left(t_0, x_0\right)$
$$
\tilde{f}_t=\tilde{f}_{x x}+\tilde{f}+\phi - \frac{(\tilde{f}+\phi)^3}{3} - \tilde{g}<0,
$$
which contradicts \eqref{equation-for-contradicition-middleproof-globalsolution}. If we are on the left side of $\partial R(T)$ all the inequalities are reversed. For the top side, if $t_0$ is the first time such that for some $x_0, \tilde{g}\left(t_0, x_0\right) \in \partial R(T)$, we again have 
$$
    \tilde{g}_t\left(t_0, x_0\right) \geqslant 0.
$$

But again by the construction of $R$ we know that at $\left(t_0, x_0\right)$
$$
\tilde{g}_t = \varepsilon (\tilde{f} + \gamma \tilde{g}) < 0,
$$
getting the same contradiction as we got on the right hand side. For the bottom face the inequalities are reversed. Hence, the solution $\widetilde{H}=(\tilde{f}, \tilde{g})$ remains in $R(T)$ for $0 \leqslant t \leqslant T$. Now, we have that there exist $c>0$ such that
$$
\|H(t)-\Phi(t)\|_{\infty} \leqslant c.
$$
Since the growth of $\Phi$ is at most exponential we get for a constant $k$ that
$$
\|H(t)\|_{\infty} \leqslant c \exp (k T), \quad 0 \leqslant t \leqslant T .
$$
\end{proof}

\section{Averaged System}
In this section, we apply averaging techniques to derive the Averaged System \eqref{partially-avg-system}. This approach has been successfully used in cases where trigonometric internal sources affect the system \cite{cerpa2023approximation,MR2881590}. However, unlike previous works that dealt with explicit expressions, which allowed the decomposition of the solution into fast and slow components, we face additional complexities here. We implement the averaging technique to demonstrate that it still works effectively even when the expression is implicit, specifically when it arises as the solution to a partial differential equation. This new application not only confirms the  robustness of the method but also extends its applicability to more challenging scenarios. Finally, we establish the existence of small contracting rectangles, which will help us assess the desired stability.
\subsection{Averaged System}
This section is devoted for the derivation of the Averaged System. To do so, we look for solutions $(v,u)$ of \eqref{centeredfhn} of the form 
\begin{equation*}
    v = \tilde{V} + w, \quad u = \tilde{U}
\end{equation*}
where $(\tilde{V},\tilde{U})$ is the slow varying part of $(v,u)$ and $w$ is the solution of the following system
\begin{equation}
    \left\{\begin{array}{ll}
         w_t - w_{xx} = \dot{h}(t)\psi(x) + \left(1-(v_0-h(t)\psi(x))^2 \right)w, & x\geq 0, t\geq 0,\\ w(0,x) = 0, & x\geq 0,\\
         w_x(t,0) = 0, & t\geq 0.
    \end{array}\right.
\end{equation}
We substitute into \eqref{centeredfhn} to get the following equations for $(\tilde{V},\tilde{U})$
\begin{equation}\label{before-averaging-system}
    \begin{aligned}
        \tilde{V}_t + w_t - \tilde{V}_{xx} - w_{xx} &= \dot{h}\psi - h\psi^{\prime \prime} + \tilde{V} + w - h\psi - \frac{(\tilde{V} + w)^3}{3} - (\tilde{V} + w)^2v_0 - (\tilde{V} + w)v_0^2  + (\tilde{V} + w)^2h\psi \\
        &+2(\tilde{V} +w)v_0h\psi + v_0^2h\psi - (\tilde{V} + w)h^2\psi^2 - v_0h^2\psi^2 + \frac{h^3\psi^3}{3} - \tilde{U}\\
        \tilde{U}_t &= \varepsilon(\tilde{V} + w -\gamma \tilde{U}) + \varepsilon h\psi
        \end{aligned}
\end{equation}
We assume that $(\tilde{V},\tilde{U})$ is slowly varying in comparison with the high-frequency term $w$. Thus, it makes sense to consider an averaging process over the period $(t-\frac{\pi}{\omega},t+\frac{\pi}{\omega})$. Setting 
\begin{equation}\label{w-average}
w_{avg}(t,x) = \dfrac{\omega}{2\pi} \Int_{t-\frac{\pi}{\omega}}^{t+\frac{\pi}{\omega}} w(x,\tau) d \tau,
\end{equation}
we get the following averaged system
\begin{equation}\label{system-avg-with-trash}
\left\{\begin{aligned}
\tilde{V}_t - \tilde{V}_{xx} =& \tilde{V} - \tilde{V}^3/3-\tilde{V}^2v_0 - \tilde{V}v_0^2 +2\tilde{V}(w)_{avg}(h\psi - v_0) -\tilde{V}^2(w)_{avg}  \\
&\qquad - (\tilde{V} + v_0)\left(\frac{A^2\psi^2}{2} + (w^2)_{avg}\right) - \dfrac{(w^3)_{avg}}{3} + (w^2h)_{avg}\psi - \tilde{U},  & & x\geq 0, t\geq 0, \\
\tilde{U}_t - d \tilde{U}_{xx} =& \varepsilon(\tilde{V} + (w)_{avg} - \gamma \tilde{U}) & & x\geq 0, t\geq 0, \\
\tilde{V}(0,x) = \overline{f}_0(x)&, \, \tilde{U}(0,x) = \overline{g}_0(x), & & x \geq 0,
\end{aligned}\right.
\end{equation}
and the system for $w_{avg}$,
\begin{equation}
    \left\{\begin{array}{ll}
         (w_{avg})_t - (w_{avg})_{xx} =  \left(1-v_0^2 + \dfrac{A^2\psi^2}{2} \right)w_{avg} + O\left(\frac{A}{\omega}\right), & x\geq 0, t\geq 0,\\ w_{avg}(0,x) = 0, & x\geq 0,\\
         (w_{avg})_x(t,0) = 0, & t\geq 0.
    \end{array}\right.
\end{equation}
For the high frequency term $w$ and $w_{avg}$ we have the following lemma that gives us important estimates for the following sections.
\begin{lemma}\label{lemma-bound-high-frquencyterm}
    (Estimate for the high frequency term.) Let $\mathcal{B} =W^{k, p}(\mathbb{R^+})$ with $k \geq 3, p \geq 1$ and let $(v_0,u_0)$ be the unique solution of \eqref{centerequation}. Then, for any $T>0 $ there exists $w \in C([0,T],\mathcal{B} )$ solution of \eqref{highfrequencyequation} and a constant $C>0$ such that 
    \begin{equation}\label{bound-for-high-frequency-term}
        \abs{w(t,x)} \leq C_1 A, \quad \text{and} \quad \abs{w_{avg}(t,x)} \leq  \dfrac{C_2A}{\omega}\quad \text{for all} \quad x\in [0,\infty), 0\leq t \leq T,
    \end{equation}
    where $C_1 = C(v_0,\|\psi\|_{\infty},\|\psi^{\prime}\|_{\infty})$ and $C_2 = C(v_0,\|\psi\|_{\infty})$.
\end{lemma}
\begin{proof}
    See Appendix \ref{appendix-proof-high-frequency}.
\end{proof}
Due to Lemma \ref{lemma-bound-high-frquencyterm} we ignore terms of order $O(A^2)$ and order $O(\frac{1}{\omega})$ on system \eqref{system-avg-with-trash} and we are able to deduce the averaged system \eqref{partially-avg-system} that looks like
\begin{equation*}
\left\{\begin{aligned}
V_t - V_{xx} =& V - V^3/3-V^2v_0 - Vv_0^2 - U ,  & & x\geq 0, t\geq 0, \\
U_t - d U_{xx} =& \varepsilon(V  - \gamma U) & & x\geq 0, t\geq 0, \\
V(0,x) = \overline{f}_0(x)&, \, U(0,x) = \overline{g}_0(x), & & x \geq 0.
\end{aligned}\right.
\end{equation*}
In the next section we prove that the system \eqref{partially-avg-system} is, in fact, a faithful approximation of the original system \eqref{centeredfhn}.
\subsection{Existence of small contracting rectangles}
In this section, we construct small contracting rectangles for the vector field associated with system \eqref{partially-avg-system}. To do this, we consider the following vector-valued function $\tilde{F}: \R^2\to \R^2,$
\begin{equation}\label{vector-field-for-average-system}
\tilde{F}((V, U)) = \left( \left(1-v_0^2\right) V-\frac{V^3}{3} -V^2v_0 - U, \varepsilon(V-\gamma U) \right).
\end{equation}
Now we present the result that addresses the existence of arbitrary small contracting rectangles.
\begin{lemma}\label{lemma-existence-smallrectangles}
Let $\varepsilon >0$, let $\beta, \gamma$ satisfy \eqref{parameters} and let $(v_0,u_0)$ be the unique solution of \eqref{centerequation}. If $(L,S)$ belongs to the following set
\begin{equation}
\left\{(L, S): 0<L<\left|v_0\right|, \quad \dfrac{1}{\gamma}<\frac{S}{L}<v_0^2-1, \quad L<\frac{\left(v_0^2-1-S / L\right)}{\left(-v_0-L / 3\right)}\right\} 
\end{equation}
then, the rectangles $R_{L,S}$ are contracting under the flow $\tilde{F}((U,V))$. Moreover, given any open neighborhood $\oo$ of $(0,0)$ there exist $(L,S)$ such that $R_{L,S} \in \oo$. 
\end{lemma}
\begin{proof}
 We want to show that the set of pairs $(L,S) \in \R^2_+$ such that the rectangle $R_{L,S}$ is contracting is non-empty. In order to do this, we have to prove
that the vector field $\tilde{F}$ is pointing inwards at each point of the boundary $\partial R_{L,S}$ . Since we know the
normal vectors at each face we can do explicit calculations as follows.
 \begin{enumerate}[align=left, labelwidth=1ex]
 \item[1. Top face.] At $U=S, V \in[-L, L]$ we have
$$
\begin{aligned}
(0,1) \cdot \tilde{F}((V, U)) & =\varepsilon(V  -\gamma S) \\
& \leq \varepsilon(L-  \gamma S) .
\end{aligned}
$$

The vector field will point inwards if $L-\gamma S<0$, or equivalently $\frac{1}{\gamma}<\frac{S}{L}$.
 \item[2. Bottom face.] At $U=-S, V \in[-L, L]$ we have
$$
\begin{aligned}
(0,-1) \cdot \tilde{F}((V, U)) & =-\varepsilon(V+\gamma S) \\
& \leq \varepsilon(L-\gamma S ) .
\end{aligned}
$$
We get the same condition as for the top face.
\item[3. Left face.] Assume $0<L<-v_0$. For $V=-L, U \in[-S, S]$ we get the following
$$
\begin{aligned}
(-1,0) \cdot \tilde{F}((V, U))= & -\left(\left(1-v_0^2\right)(-L)-v_0 L^2-\dfrac{(-L^3)}{3}-U\right)\\
\leq & -\left(v_0^2-1-\dfrac{S}{L}\right) L-\left(-v_0+L / 3\right) L^2\\
\leq & -\left(v_0^2-1-\frac{S}{L}\right) L .
\end{aligned}
$$
We want $L$ and $S$ such that the right-hand side becomes negative. This will occur if the following equation is satisfied.
\begin{equation}\label{proof-smallrectangles-condition-leftface}
    v_0^2-1-\frac{S}{L}> 0.
\end{equation}

Putting together conditions on top/bottom faces and \eqref{proof-smallrectangles-condition-leftface} we see that the sides must satisfy
\begin{equation}\label{proof-smallrectangles-condition-leftface-and-topbottomface}
0<L<-v_0 \text { and } \frac{1}{\gamma}<\frac{S}{L}<v_0^2-1
\end{equation}
\item[4. Right face.] Assume $0<L<-v_0$. For $V=L, U \in[-S, S]$ we get the following
$$
\begin{aligned}
(1,0) \cdot \tilde{F}((V, U)) & =\left(1-v_0^2\right) L-v_0 L^2-L^3 / 3-U\\
& \leq\left(-v_0-L / 3\right) L^2-\left(v_0^2-1-\frac{S}{L}\right) L.
\end{aligned}
$$

Since $\left(-v_0-L / 3\right) L^2>0$ the right-hand side in the equation above will be negative if
\begin{equation}\label{proof-smallrectangles-condition-rightface}
    L<\frac{\left(v_0^2-1-\frac{S}{L}\right)}{\left(-v_0-L / 3\right)} .
\end{equation}
\end{enumerate}
Combining the conditions from top/bottom sides, \eqref{proof-smallrectangles-condition-leftface}, \eqref{proof-smallrectangles-condition-leftface-and-topbottomface} and \eqref{proof-smallrectangles-condition-rightface} we obtain the following set of sides $(L, S)$ for which $R_{L, S}$ is a contracting rectangle under $\tilde{F}((V, U))$,
$$
\left\{(L, S): 0<L<\left|v_0\right|, \quad \frac{1}{\gamma}<\frac{S}{L}<v_0^2-1, \quad L<\frac{\left(v_0^2-1-S / L\right)}{\left(-v_0-L / 3\right)}\right\} .
$$
For the last part, we see that Lemma \ref{parameters-lemma} gives us $v_0^2-1-\frac{1}{\delta \gamma}>0$ for $\delta>0$. But since we are restricting to $\delta < 1/16$ (see Lemma \ref{lemma-paramters-not-restrictive}) we have $v_0^2-1-\frac{1}{\gamma}>0$. Then, let us consider $\mathcal{O}$ an arbitrary neighborhood of $(0,0)$ and take $\epsilon>0$ small so that $\frac{1+\epsilon}{\gamma}<v_0^2-1$ and take $S_0=\frac{1+\epsilon}{\gamma} L_0$. Next, take $0<L_0<\left|v_0\right|$ small enough so that $\left( \pm L_0, \pm S_0\right) \in \mathcal{O}$ and
$$
0<L_0<\frac{\left(v_0^2-1-\frac{1+\epsilon}{\gamma}\right)}{\left(-v_0-L_0 / 3\right)},
$$
which implies that the pair $(L, S)=\left(L_0, \frac{1+\epsilon}{\gamma} L_0\right)$ satisfies $R_{L, S} \subset \mathcal{O}$.
\end{proof}
\subsection{Proof Theorem \ref{stability-avg-system}}
We show that if $U_0 = (f(0, x)-v_0, g(0, x)-u_0)$ lies within a contracting rectangle $R$ for every $x>0$, then $\tilde{U} = (V(t, x), U(t, x))$ stays in $R$ for every $x>0$ and $t>0$. This is done defining the norm in the space 
$$X=\left\{\vec{v} \in C\left(\mathbb{R}^+ ; \mathbb{R}^2\right): \vec{v}\right. \,\, \text{is a continuous functions that converge to} \, (0,0) \, \text{when} \, \left.x \rightarrow \infty\right\},
$$
given by
$$
\|\vec{v}\|_X=\sup _{x \in \mathbb{R}} \inf \left\{r>0: \vec{v}(x) \in r R_{L, S}\right\}.
$$
By Lemma \ref{lemma-existence-smallrectangles} there exists small contracting rectangles R for the vector field \eqref{vector-field-for-average-system} contained in an arbitrary open neighborhood of $(0,0)$ $\mathcal{O}$. So, since we are assuming that $U_0$ lies inside R, for every $x>0$, we have $\|U_0\|_X < 1$ and by Theorem \ref{theorem-local-solution} there exists $U \in C\left(\left[0, t_0\right];\mathcal{B}\times \mathcal{B}\right) $. We claim that $\|U(t)\|_X < 1$ for $0<t<t_0$. If this is not true, then there exists a time $\overline{t} = \inf\{t\in (0,t_0)| \|U(t)\|_X = 1 \}$. By \cite[Lemma 3.4]{SCHONBEK1978119}, we have that
$$
\overline{D} \|U(\overline{t})\|_X < 0,
$$
where $\overline{D}$ represents the right Dini derivative. This implies that for any $t \in (\overline{t} - \epsilon, \overline{t}]$ we have $\|U(t)\|_X >1$, but since $\|U_0\|_X < 1$ this mean that there exists $t^{\star} < \overline{t}$ such that $\|U(t^{\star})\|_X =1$  which is a contradiction on the definition of $\overline{t}$. Thus, the estimate $\|U(t)\|_X<1$ holds for $t \in[0, t_0 ]$. Now, since the estimate is the sup norm estimate we can extend $U$ from a local solution to a global solution with $\|U(t)\|_X<1$ for all $t \geqslant 0$. Which is equivalent to $\tilde{U} = (V(t, x), U(t, x))$ stays in $R$ for every $x>0$ and $t>0$. This concludes the proof of Theorem \ref{stability-avg-system}
\section{Approximation result}
We now prove the approximation result.
\begin{proposition}\label{equation-approximation-error}
(Equation for the approximation error.) Let $\mathcal{B} =W^{k, p}(\mathbb{R^+})$ with $k \geq 3, p \geq 1$ and let $(v,u)$ the solution of the system \eqref{centeredfhn} and $(V,U)$ the solution of the Averaged System \eqref{partially-avg-system} with the same initial data. Then, the approximation error given by $E_v = v - V - w$ and $E_u = u - U$ satisfy
\begin{equation}\label{nonlinear-error-equation}
\left\{\begin{aligned}
\partial_t E_v-\partial_x^2 E_v & =\left(1-\left(v_0+V\right)^2+\varphi_1\right) E_v+\varphi_2 E_v^2-\frac{1}{3} E_v^3-E_u+\varphi_3, \\
\partial_t E_u-d\partial_x^2 E_u & =\varepsilon\left(E_v-\gamma E_u\right)+\varepsilon (w+h\psi), \\
E_v(0)=0, & E_u(0)=0,
\end{aligned}\right.
\end{equation}
where $w$ is the solution of \eqref{highfrequencyequation} and 
\begin{align}
    \varphi_1 &= -2(v_0+V)(w-h\psi) - (w-h\psi)^2,\\
    \varphi_2 &= -V - w - v_0 + h\psi,\\
    \varphi_3 &= -h\psi^{\prime \prime} - h\psi - Vw^2 - \frac{w^3}{3}-V^2w- w^2v_0 -2Vwv_0\\ &+V^2h\psi+ w^2h\psi+2Vwh\psi +2Vv_0h\psi +v_0^2h\psi\\
    &- Vh^2\psi^2- v_0h^2\psi^2 + \frac{h^3\psi^3}{3}.
\end{align}
\end{proposition}
\begin{proof}
    It is immediate from taking the difference between \eqref{partially-avg-system} and \eqref{centeredfhn}.
\end{proof}
\subsection{Linear estimate of the error}
We prove an approximation result for the linear problem
\begin{equation}\label{linear-error-equation}
\left\{\begin{aligned}
\partial_t F_v-\partial_x^2 F_v & =\left(1-\left(v_0+V\right)^2+\varphi_1\right) F_v-F_u+\varphi_3, \\
\partial_t F_u-d \partial_x^2 F_u & =\varepsilon F_v-\varepsilon \gamma F_u+\varepsilon (w+h\psi), \\
F_v(x, 0)=0, & F_u(x, 0)=0,
\end{aligned}\right.
\end{equation}
\begin{lemma}\label{lemma-linear-error}
(Linear estimate of the error). Let $\mathcal{B} =W^{k, p}(\mathbb{R^+})$ with $k \geq 3, p \geq 1$ and let $\beta, \gamma$ satisfy \eqref{parameters}, let $\varepsilon>0$, and let $\left(v_0, u_0\right)$ be given by \eqref{centerequation}. Suppose that the parameters $A, \gamma$, the lift off function $\psi$ and the solution $(V, W)$ of Averaged System \eqref{partially-avg-system} satisfy, for some $T>0$, the estimate
\begin{equation}\label{condition-for-contraction}
\alpha(T):=\frac{\left\|v_0^2-\left(v_0+V\right)^2\right\|_Y}{\sqrt{v_0^2-1}}+\frac{M^2}{\sqrt{v_0^2-1}}+\frac{2 M\left\|v_0+V\right\|_Y}{\sqrt{v_0^2-1}}+\frac{1}{\gamma\sqrt{v_0^2-1}}<1, \quad t \in[0, T],
\end{equation}
where $M\geq \left\| w- h\psi\right\|_Y$. Then, there exist $\left(F_v, F_u\right) \in C([0, T], \mathcal{B} \times \mathcal{B})$ solution of the system \eqref{linear-error-equation} with initial data $\left(F_v(0), F_u(0)\right)=(0,0)$, and a constant $C=C(v_0,\gamma, \norm{\psi}_{\infty},\norm{\psi^{\prime}}_{\infty})$ such that
\begin{equation}\label{bounds-linear-equation}
\left|F_v(t, x)\right| \leq AC, \quad\left|F_u(t,x)\right| \leq AC,
\end{equation}
for all $x \in [0,\infty), 0 \leq t \leq T$.    
\end{lemma}
\begin{remark}
    Before we begin with the proof, to simplify the notation let us denote $G_{i,\rho} = G_{\rho}$, where $G_{i,\rho}$ is the Heat Kernel on the half line with diffusion coefficient $\rho$, we will just write $G$ for the case $\rho = 1$. See \ref{A.1} for more details.\end{remark}
\begin{proof}
The idea of the proof is to consider an iterative approximation of \eqref{linear-error-equation}. Multiplying the first equation in \eqref{linear-error-equation} by $e^{-\left(1-v_0^2\right) t}$, we get
$$
\partial_t e^{-\left(1-v_0^2\right) t} F_v-\partial_x^2 e^{-\left(1-v_0^2\right) t} F_v=\left(\left(v_0^2-\left(v_0+V\right)^2\right)+\varphi_1\right) e^{-\left(1-v_0^2\right) t} F_v-e^{-\left(1-v_0^2\right) t} F_u+e^{-\left(1-v_0^2\right) t} \varphi_3.
$$

By virtue of Duhamel's principle, we get
$$
F_v(t,x)=G(t) * F_v(0,\cdot)+\int_0^t e^{\left(1-v_0^2\right)(t-\tau)} G(t-\tau) *\left(\left(\left(v_0^2-\left(v_0+V\right)^2\right)+\varphi_1\right) F_v-F_u+\varphi_3\right) d \tau .
$$
For the second equation, we obtain
$$
F_u(t,x)=G_d(t) * F_u(0,\cdot)+\varepsilon \int_0^t e^{-\varepsilon \gamma(t-\tau)} G_d(t-\tau) *\left(F_v+w+h\psi\right) d \tau .
$$

Let us consider the following iterative procedure. Set $F_v^{(0)}=0, F_u^{(0)}=0$ and define
$$
\begin{aligned}
& F_v^{(k+1)}(t,x)=G(t) * F_v(0,\cdot)+\int_0^t e^{\left(1-v_0^2\right)(t-\tau)} G(t-\tau) *\left(\left(\left(v_0^2-\left(v_0+V\right)^2\right)+\varphi_1\right) F_v^{(k)}-F_u^{(k)}+\varphi_3\right) d \tau \\
& F_u^{(k+1)}(t,x)=G_d(t) * F_u(0,\cdot)+\varepsilon \int_0^t e^{-\varepsilon \gamma(t-\tau)} G_d(t-\tau) *\left(F_v^{(k+1)}+w+h\psi\right) d \tau,
\end{aligned}
$$
where $\varphi_1, \varphi_3$ and $w$ are defined in \eqref{equation-approximation-error}, and where $\left(F_v^{(k)}, F_u^{(k)}\right) \in B C^0(\R^+) \times B C^0(\R^+)$ imply that $\left(F_v^{(k+1)}, F_u^{(k+1)}\right) \in B C^0(\R^+) \times B C^0(\R^+)$. The next step is to look at the convergence of the sequences $\left\{F_v^{(k)}\right\},\left\{F_u^{(k)}\right\}$ in the space $C\left([0, T] ; B C^0(\R^+)\right)$. Let us estimate the difference between two consecutive terms. For $\left\{F_v^{(k)}\right\}$, it holds

\begin{multline*}
\left\|F_v^{(k+1)} - F_v^{(k)}\right\|_Y \\
\leq \sup_{0 < t < T} \int_0^t e^{(1 - v_0^2)(t - \tau)} 
\left\| G(t - \tau) * \left( \left|v_0^2 - (v_0 + V(\tau,\cdot))^2 \right| 
+ \left|\varphi_1(\tau,\cdot)\right| \right) \right\|_{L_x^\infty} d\tau 
\left\| F_v^{(k)} - F_v^{(k-1)} \right\|_Y \\
\quad + \sup_{0 < t < T} \int_0^t e^{(1 - v_0^2)(t - \tau)} d\tau 
\left\| F_u^{(k)} - F_u^{(k-1)} \right\|_Y \\
\leq \sup_{0 < t < T} \int_0^t e^{(1 - v_0^2)(t - \tau)} 
\left\| G(t - \tau) * \left( \left|v_0^2 - (v_0 + V(\tau,\cdot))^2 \right| 
+ \left|\varphi_1(\tau,\cdot)\right| \right) \right\|_{L_x^\infty} d\tau 
\left\| F_v^{(k)} - F_v^{(k-1)} \right\|_Y \\
\quad + \sup_{0 < t < T} \int_0^t e^{(1 - v_0^2)(t - \tau)} 
\frac{1}{\gamma} d\tau \left\| F_v^{(k)} - F_v^{(k-1)} \right\|_Y.
\end{multline*}

Thus,

$$
\begin{aligned}
\left\|G(t-\tau) *\left(\left|v_0^2-\left(v_0+V(\tau,\cdot)\right)^2\right|+\left|\varphi_1(\tau,\cdot)\right|\right)\right\|_{L^{\infty}_x} & \leq\dfrac{\left\|\left|v_0^2-\left(v_0+V(\tau,\cdot)\right)^2\right|+\mid \varphi_1(\tau,\cdot)\right\|_{L^{\infty}_x}}{\sqrt{\pi(t-\tau)}} \\
& \leq\dfrac{\left\|v_0^2-\left(v_0+V(\tau,\cdot)\right)^2\right\|_{L_x^{\infty}}+\left\|(w(\tau,\cdot)-h\psi)^2\right\|_{L^{\infty}_x}}{\sqrt{\pi(t-\tau)}}\\
&+\dfrac{2\left\|\left(v_0+V\right) (w-h\psi)\right\|_{L^{\infty}_x}}{\sqrt{\pi(t-\tau)}} \\
& \leq\dfrac{\left\|v_0^2-\left(v_0+V(\tau,\cdot)\right)^2\right\|_Y+M^2+2 M\left\|v_0+V\right\|_Y}{\sqrt{\pi(t-\tau)}}
\end{aligned}
$$
which replacing above and using that $v_0^2 - 1 >0$ implies
\begin{equation*}
\begin{aligned}
\left\|F_v^{(k+1)}-F_v^{(k)}\right\|_Y & \leq\left\|F_v^{(k)}-F_v^{(k-1)}\right\|_Y\left(\left\|v_0^2-\left(v_0+V(\tau,\cdot)\right)^2\right\|_Y+M^2+2 M\left\|v_0+V\right\|_Y+\dfrac{1}{\gamma}\right) \sup _{0 \leq t \leq T} \int_0^t \dfrac{e^{\left(1-v_0^2\right)(t-\tau)}}{\sqrt{\pi(t-\tau)}} d \tau \\
& \leq\left\|F_v^{(k)}-F_v^{(k-1)}\right\|_Y \frac{\left\|v_0^2-\left(v_0+V(\tau,\cdot)\right)^2\right\|_Y+M^2+2 M\left\|v_0+V\right\|_Y+\dfrac{1}{\gamma}}{\sqrt{v_0^2-1}}.
\end{aligned}
\end{equation*}
In the last inequality we used that $$\Int_0^t \dfrac{e^{\left(1-v_0^2\right)(t-\tau)}}{\sqrt{\pi(t-\tau)}} d \tau = \dfrac{\sqrt{\pi}  erf\left(\sqrt{(v_0^2-1)t}\right)}{\sqrt{v_0^2-1}}$$ and the fact the error function $|erf(z)| \leq 1$ for all $z\in \R$. Defining $$\alpha(T)=\left(\left\|v_0^2-\left(v_0+V(\tau,\cdot)\right)^2\right\|_Y+M^2+2 M\left\|v_0+V\right\|_Y+\dfrac{1}{\gamma}\right) /\sqrt{v_0^2-1}$$ we obtained that
$$
\left\|F_v^{(k+1)}-F_v^{(k)}\right\|_Y \leq \alpha(T)\left\|F_v^{(k)}-F_v^{(k-1)}\right\|_Y .
$$

Analogously, for $\left\{F_u^{(k)}\right\}$, we get
$$
\left\|F_u^{(k+1)}-F_u^{(k)}\right\|_Y \leq \varepsilon \sup _{0<t<T} \int_0^t \dfrac{e^{-\varepsilon \gamma(t-\tau)}}{\sqrt{d\pi(t-\tau)}} d \tau\left\|F_v^{(k+1)}-F_v^{(k)}\right\|_Y \leq \dfrac{1}{\gamma\sqrt{d}}\left\|F_v^{(k+1)}-F_v^{(k)}\right\|_Y .
$$
Now, in order to state the convergence of the sequences $\left\{F_v^{(k)}\right\},\left\{F_u^{(k)}\right\}$ in $C\left([0, T] ; B C^0(\R^+)\right)$, use that for $m>n$
$$
\begin{gathered}
\left\|F_v^{(m)}-F_v^{(n)}\right\|_Y \leq \sum_{k=n}^{m-1}\left\|F_v^{(k+1)}-F_v^{(k)}\right\|_Y \leq \sum_{k=n}^{m-1} \alpha(T)^k\left\|F_v^{(1)}-F_v^{(0)}\right\|_Y, \\
\left\|F_u^{(m)}-F_u^{(n)}\right\|_Y \leq \sum_{k=n}^{m-1}\left\|F_u^{(k+1)}-F_u^{(k)}\right\|_Y \leq \sum_{k=n}^{m-1} \dfrac{1}{\gamma\sqrt{d}}\left\|F_v^{(k+1)}-F_v^{(k)}\right\|_Y \leq \sum_{k=n}^{m-1} \dfrac{1}{\gamma\sqrt{d}} \alpha(T)^k\left\|F_v^{(1)}-F_v^{(0)}\right\|_Y .
\end{gathered}
$$
Condition \eqref{condition-for-contraction} guarantees that we have a contractive mapping hence the sequence $\left\{\left(F_v^{(k)}, F_{u}^{(k)}\right)\right\}_k$ converges in $C\left([0, T] ; B C^0(\R^+) \times B C^0(\R^+)\right)$ and Theorem \ref{proposition-for-uniqueness} implies that $\left(F_v, F_u\right) \in C([0, \infty), \mathcal{B} \times \mathcal{B})$. To prove \eqref{bounds-linear-equation}, which can be obtained by bounding in a proper manner $\left\|F_v^{(1)}-F_v^{(0)}\right\|_Y$, we write
\begin{align*}
\left\|F_v^{(1)}-F_v^{(0)}\right\|_Y&=\left\| \Int_0^t e^{\left(1-v_0^2\right)(t-\tau)} G(t-\tau) * \varphi_3 d \tau \right\|_Y \\
&\leq \sup _{0 \leq t \leq T} \int_0^t e^{\left(1-v_0^2\right)(t-\tau)}\left\| G(t-\tau) \ast \varphi_3(\tau, \cdot)\right\|_{L^{\infty}_x} d \tau  \\
&\leq  \sup _{0 \leq t \leq T} \int_0^t e^{\left(1-v_0^2\right)(t-\tau)} d \tau \left\| \varphi_3(\tau,\cdot) \right\|_{L^{\infty}_x}.
\end{align*}
Using estimate \eqref{bound-for-high-frequency-term} and Theorem \ref{stability-avg-system} we get
\begin{equation}
 \left\|F_v^{(1)}-F_v^{(0)}\right\|_Y \leq AC,   
\end{equation}
where $C=C(v_0,\gamma, \norm{\psi}_{\infty},\norm{\psi^{\prime}}_{\infty})$. Finally, using the convergence of the sequence $\left\{F_v^{(k)}\right\}$ and that $\alpha=\alpha(T)<1$ we conclude
$$
\begin{aligned}
\left\|F_v-F_v^{(0)}\right\|_Y & \leq\left\|F_v-F_v^{(N+1)}\right\|_Y+\sum_{k=0}^N\left\|F_v^{(k+1)}-F_v^{(k)}\right\|_Y \\
& \leq\left\|F_v-F_v^{(N+1)}\right\|_Y+\sum_{k=0}^N \alpha^k\left\|F_v^{(1)}-F_v^{(0)}\right\|_Y \\
& \leq\left\|F_v-F_v^{(N+1)}\right\|_Y+\frac{AC\alpha}{1-\alpha}.
\end{aligned}
$$

Therefore, by taking the limit as $N \rightarrow \infty$ we get the first estimate in \eqref{bounds-linear-equation}. Analogously for $F_u$ we get
$$
\begin{aligned}
\left\|F_u-F_u^{(0)}\right\|_Y & \leq\left\|F_u-F_u^{(N+1)}\right\|_Y+\sum_{k=0}^N\left\|F_u^{(k+1)}-F_u^{(k)}\right\|_Y \\
& \leq\left\|F_u-F_u^{(N+1)}\right\|_Y+\sum_{k=0}^N \dfrac{1}{\gamma}\left\|F_u^{(k+1)}-F_u^{(k)}\right\|_Y \\
& \leq\left\|F_u-F_u^{(N+1)}\right\|_Y+\dfrac{1}{\gamma\sqrt{d}} \frac{AC\alpha}{1-\alpha}.
\end{aligned}
$$
Taking the limit as $N \rightarrow \infty$ we get the second part of \eqref{bounds-linear-equation}.
\end{proof}
\begin{remark}
The proof of Lemma \ref{lemma-linear-error} has been established for every $d > 0$ using the Gaussian kernel $G_d$. But, if we replace $G_d$ with the limiting kernel associated with diffusion at $d = 0$, namely the Dirac delta $G_0 = \delta(x)$, the structure of the proof remains unchanged. Since no additional assumptions on $d > 0$ were necessary beyond the specific form of $G_d$, and the arguments remain valid under the substitution $G_0 = \delta(x)$, we conclude that the proof holds for all $d \geq 0$.
\end{remark}
The next result establishes that the condition \eqref{condition-for-contraction} in the previous lemma is not too restrictive.
\begin{lemma}\label{lemma-paramters-not-restrictive}
Let $\beta, \gamma$ and $\delta$ satisfy \eqref{parameters} and let $\varepsilon>0$. Suppose that
\begin{enumerate}
    \item[(i)] $\|V\|_Y \leq \min \left\{1, \dfrac{1}{\sqrt{\gamma}} \frac{1}{1+2 \max \{\sqrt{3}, \beta\}}\right\}$,
    \item[(ii)] $M  \leq \min \left\{\frac{1}{\gamma^{1/4}}, \frac{1}{2 \sqrt{\gamma}(1+\max \{\sqrt{3}, \beta)\}}\right\}$.
\end{enumerate}
Then, condition \eqref{condition-for-contraction} in Lemma \ref{lemma-linear-error} is satisfied.
\end{lemma}
\begin{proof}
    Using the bounds in $\frac{1}{v_0^2-1}$ and $\left|v_0\right|$ provided by Lemma \ref{parameters-lemma} and our hypothesis we get
$$
\begin{aligned}
\frac{\left\|v_0^2-\left(v_0+V\right)^2\right\|_Y}{\sqrt{v_0^2-1}} & \left.=\frac{\left\|V\left(V+2 v_0\right)\right\|_Y}{\sqrt{v_0^2-1}} \leq \sqrt{\delta \gamma}\|V\|_Y\left(\|V\|_Y+2 \max \{\sqrt{3}, \beta\}\right\}\right) \leq \sqrt{\delta}, \\
\frac{M^2}{\sqrt{v_0^2-1}} & \leq  M^2 \sqrt{\delta \gamma} \leq \sqrt{\delta}, \\
\frac{2 M\left\|v_0+V\right\|_Y}{\sqrt{v_0^2-1}} & \leq 2 M\left(\|V\|_Y+\max \{\sqrt{3}, \beta\}\right) \sqrt{\delta \gamma} \leq \sqrt{\delta}, \\
\dfrac{1}{\sqrt{\gamma}} \frac{1}{\sqrt{v_0^2-1}} & \leq \sqrt{\delta} .
\end{aligned}
$$

We obtain that for all $0<\delta<1 / 4$ condition \eqref{condition-for-contraction} is satisfied.
\end{proof}
\subsection{Nonlinear estimate of the error}
The next lemma shows the error estimate for the non-linear equation \eqref{nonlinear-error-equation}. We omit the proof since it is analogous to the proof of \cite[Lemma 15]{cerpa2023approximation}.
\begin{lemma}\label{lemma-nonlinear-error}
(Non-linear estimate for the error equation.) Let $\mathcal{B} =W^{k, p}(\mathbb{R^+})$ with $k \geq 3, p \geq 1$,$\varepsilon>0$, let $\gamma, \beta$ satisfy \eqref{parameters} and let $\left(v_0, u_0\right)$ be given by \eqref{centerequation}. For some $T>0$ let $(V, U) \in C([0, T] ; \mathcal{B} \times \mathcal{B})$ be a solution of \eqref{partially-avg-system}, let $\varphi_1, \varphi_2$ be given by Lemma \ref{lemma-linear-error} and let $\left(F_v, F_u\right) \in C([0, T] ; \mathcal{B} \times \mathcal{B})$ be the corresponding solution of \eqref{linear-error-equation}.
Given $\mu>0$ there exist constants $C_1(\mu), C_2(\mu), C_3(\mu)>0$ such that if
\begin{itemize}
\item $\|V\|_Y \leq C_1(\mu)$,
\item $\left\|\varphi_1\right\|_Y \leq C_2(\mu)$,
\item $\left\|\varphi_2\right\|_Y \leq\left|v_0\right|+C_1(\mu)\left( 2\norm{w}_Y +1 \right) +\norm{h\psi}_Y$,
\item $\left|F_v(t,x)\right| \leq C_3(\mu),\left|F_u(t,x)\right| \leq \mu / 2$, for $x \in [0,\infty), 0 \leq t \leq T$,
\end{itemize}
then, there exists a unique solution $\left(E_v, E_u\right) \in C([0, T] ; \mathcal{B} \times \mathcal{B})$ of \eqref{nonlinear-error-equation} that also satisfies the estimate
$$
\left|E_v(t,x)\right| \leq \mu, \quad\left|E_u(t,x)\right| \leq \mu, \quad \forall x \in [0,\infty), \quad 0 \leq t \leq T .
$$   
\end{lemma}
\subsection{Proof of Theorem \ref{theorem-approximation}}
\begin{proof}
Given $\mu>0$, we choose $C_1(\mu), C_2(\mu)$, $C_3(\mu)>0$ as in Lemma \ref{lemma-nonlinear-error} and fix $T>0$. In order to use the above results we will need $\|V\|_Y$ and $A\norm{\psi}_Y$ to be small enough. To make this more precise, assume we have $\|V\|_Y<M_1$ and $A\norm{\psi}_Y + \norm{w}_Y<N_1$. If $M_1, N_1$ were small, as indicated in Lemma \ref{lemma-paramters-not-restrictive}, then condition \eqref{condition-for-contraction} would be fulfilled and Lemma \ref{lemma-linear-error} would imply
$$
\left|F_v(t,x)\right| \leq C_3(\mu), \quad\left|F_u(t,x)\right| \leq \mu / 2, \quad \forall x \in [0,\infty), t \in[0, T].
$$
Additionally, we would also have the bounds
\begin{enumerate}
\item[$\bullet$] $\left|(v_0^2-\left(v_0+V\right)^2 \right| \leq M_1\left(2\left|v_0\right|+M_1\right)$,
\item[$\bullet$] $\left|h\psi\right| \leq N_1$,
\item[$\bullet$] $\left|\varphi_1\right|=\left|2(v_0+V)(w-h\psi) + (w-h\psi)^2\right| \leq N_1\left(N_1+2\left(\left|v_0\right|+M_1\right)\right)$,
\item[$\bullet$] $\left|\varphi_2\right|=\left|v_0+V(1+2w) + h\psi\right| \leq\left|v_0\right|+M_1(1+2N_1)+N_1$.
\end{enumerate}
We will choose $M_1, N_1$ small enough such that whenever $\|V\|_Y<M_1$ and $A\norm{\psi}_Y +  \norm{w}_Y<N_1$ then we also have $\|V\|_Y \leq C_1(\mu),\left|\varphi_1\right| \leq C_2(\mu)$ and $\left|\varphi_2\right| \leq\left|v_0\right|+C_1(\mu)\left( 2\norm{w}_Y +1 \right) +\norm{h\psi}_Y$. With all these choices taken care of, we can put all the previous results together.\\

Given an open neighborhood $\mathcal{O}$ of $(0,0)$ we denote $\hat{\mathcal{O}}=\mathcal{O} \cap\left(-M_1, M_1\right) \times(-1,1)$ and let $R$ be the rectangle given by Theorem \ref{stability-avg-system} corresponding to $\mathcal{O}$. Set the rectangle in Theorem \ref{theorem-approximation} to be this rectangle $R$ and set $N$ in Theorem \ref{theorem-approximation} to be $N_1$. Next, by assumption, the initial data $\left(f_0-v_0, g_0-u_0\right) \in \mathcal{B} \times \mathcal{B}$ lies within the rectangle $R$ and from Theorem \ref{stability-avg-system} we know that there exists a unique solution $(V, U) \in C([0, \infty) ; \mathcal{B} \times \mathcal{B})$ of the system \eqref{partially-avg-system} with initial data $\left(f_0-v_0, g_0-u_0\right)$, which also satisfies
$$
(V(t,x), U(t,x)) \in R, \quad \forall x \in \mathbb{R}^+, \quad t \geq 0 .
$$

In particular, by construction of $R$, we have that $|V(t,x)| \leq M_1, \forall x \in \mathbb{R}, t>0$. Fix $T>0$, from the choices of $M_1, N_1$, the system \eqref{linear-error-equation} has a unique solution $\left(F_v, F_u\right) \in C([0, T] ; \mathcal{B} \times \mathcal{B})$ by Lemma \ref{lemma-linear-error}. We also verify the hypotheses of Lemma \ref{lemma-nonlinear-error} , so we conclude there exists a unique solution $\left(E_v, E_u\right) \in C([0, T] ; \mathcal{B} \times \mathcal{B})$ of \eqref{nonlinear-error-equation} which satisfies
\begin{equation}\label{bounds-for-proof-theorem-approximation}
\left|E_v(t,x)\right| \leq \mu, \quad\left|E_u(t,x)\right| \leq \mu, \quad \forall x \in \mathbb{R}^+, t \in[0, T] .
\end{equation}
Finally, using the definition of the equation \eqref{nonlinear-error-equation} for the approximation error
$$
v=V+E_v+w, \quad u=U+E_u, \quad x \in \mathbb{R}^+, t \in[0, T],
$$
we conclude that there exists a unique solution $(v, u) \in C([0, T] ; \mathcal{B} \times \mathcal{B})$ of the system \eqref{ecuacionfhninternal} with initial data $\left(\bar{f}_0, \bar{g}_0\right)=\left(f_0-v_0, g_0-w_0\right)$ which satisfies \eqref{bounds-for-proof-theorem-approximation}. Returning to the original variables, there exists a unique solution $(f, g)$ of \eqref{ecuacionfhnneuman} with initial data $\left(f_0, g_0\right)$ that satisfies $\left(f-v_0, g-u_0\right) \in C([0, T] ; \mathcal{B} \times \mathcal{B})$, and that additionally satisfies the estimate
$$
\left|f+h\psi-v_0-V-w\right| \leq \mu, \quad\left|g-u_0-U\right| \leq \mu, \quad \forall x \in \mathbb{R}^+, t \in[0, T] .
$$

Since $T>0$ was arbitrary, this concludes the proof.
\end{proof}

\section{Appendix}
\subsection{Proof of Lemma \ref{lemma-of-large-rectangle}}\label{appendix-proof-lemma-largerectangle}
For $H= (f,g)$, we have to verify that the vector field
$$
F_\phi(H,x,t)=\left(f+\phi - \frac{(f+\phi)^3}{3} - g, \varepsilon (f + \gamma g)\right) 
$$
is pointing inwards at each point of the boundary of $R.$ So we do the explicit calculations since we can write the outward normal for each face of R.
\begin{enumerate}[align=left, labelwidth=1ex]
\item[1. Top Face.] At $g=S, f \in[-L, L]$ we have
$$
\begin{aligned}
(0,1) \cdot F_\phi(H,x,t) & =\varepsilon(f-\gamma S) \\
& \leq \varepsilon(L-\gamma S) .
\end{aligned}
$$

Therefore, the vector field will point inwards if 
\begin{equation}\label{largerectangle-condition-topface}
\frac{1}{\gamma}<\frac{S}{L}.
\end{equation}
\item[2. Bottom Face.] At $g=-S, f \in[-L, L]$ we have
$$
\begin{aligned}
(0,-1) \cdot F_\phi(H,x,t) & =-\varepsilon(f+\gamma S) \\
& \leq \varepsilon(L-\gamma S) .
\end{aligned}
$$

Therefore, we get the same condition as for the top face.
\item[3. Left Face.] At $f=-L, g \in[-S, S]$ we get the following
\begin{align*}
(-1,0) \cdot F_\phi(H,x,t) &\leq -\left( -L + \phi + L^3/3 - L^2\phi + L\phi^2 - \phi^3/3 - S \right)\\
&\leq L + \tilde{M} - \dfrac{L^3}{3} + L^2\tilde{M} +\dfrac{\tilde{M}^3}{3} +S.
\end{align*}
We want to choose $L$ and $S$ such that the right-hand side is negative, then the sides of $R$ must satisfy 
\begin{equation}\label{largerectangle-condition-leftface}
 L + \tilde{M}  +L^2\tilde{M}  + \tilde{M}^3/3 + S < L^3/3.
\end{equation}
\item[4. Right Face.] At $f=L, g \in[-S, S]$ we get the following
\begin{align*}
(1,0) \cdot F_\phi(H,x,t) &= L + \phi - L^3/3 -L^2\phi - L\phi^2 - \phi^3/3 - g\\
&\leq L + \tilde{M} - L^3/3 +L^2\tilde{M} - L\phi^2 + \tilde{M}^3/3 + S\\
&\leq L + \tilde{M} - L^3/3 +L^2\tilde{M}  + \tilde{M}^3/3 + S.
\end{align*}
Therefore, we get the same condition as for the left face.
\end{enumerate}
Combining \eqref{largerectangle-condition-topface} and \eqref{largerectangle-condition-leftface} we obtain the following set of sides $(L,S)$ for which $R$ is contracting
\begin{equation}
Q = \left\{(L, S): \frac{1}{\gamma}<\frac{S}{L}, \quad L + \tilde{M}  +L^2\tilde{M}  + \tilde{M}^3/3 + S < L^3/3\right\}.
\end{equation}
Where $\gamma$ is taken according to Lemma \ref{parameters-lemma}. Given $\gamma$ and $\tilde{M}$, big values of $S$ and $L$ satisfy both restrictions given in $Q$, this address the existence of large contracting rectangles.
\subsection{Proof of Lemma \ref{lemma-bound-high-frquencyterm}}\label{appendix-proof-high-frequency}
First, let us remember the heat kernel with diffusion coefficient $\rho$
\begin{equation}
G_{\rho}(x, y, t):=\frac{1}{2 \sqrt{\rho \pi t}} e^{-\frac{(x-y)^2}{4\rho t}}.
\end{equation}
Applying an even extension and via the Duhamel principle we can see that the solution is given by
\begin{equation}
w(t, x)=\int_0^t \int_0^{\infty} G(x, y, t-\tau) f(\tau,y) d y d \tau+\int_0^{\infty} G(x, y, t) w(0,y) d y,
\end{equation}
where $f$ is the forcing term and $G(x,y,t) = \tilde{G_1}(x-y,t) + \tilde{G_1}(x+y,t)$. So, if we want to find an appropriate bound for $w$, solution of \eqref{highfrequencyequation}, we have to prove for $f\in C(\R^+ \times [0,T])$ there exist $C>0$ such that the integral term satisfies
\begin{equation}
    \left\|\int_0^t \int_0^{\infty} G(x, y, t-\tau) f(\tau,y) d y d \tau\right\|_Y \leq C.
\end{equation}
Moreover, for $\omega >>1$ we want to find $C>0$ independent of $\omega$ such that the following holds
\begin{equation}
\left\|\int_0^t \int_0^{\infty} \frac{e^{-\frac{|x-y|^2}{d(t-s)}+\left(1-v_0^2\right)(t-s)}}{(4 \pi(t-s))^{1 / 2}} f(s,y) e^{i \omega s} d y d s\right\|_Y \leq C \dfrac{1}{\omega}\left(\left\|\partial_t f\right\|_Y+\left\|\partial_x f\right\|_Y+\|f\|_Y\right).
\end{equation}
So considering the change of variables $s \rightarrow (t-\tau), (x-y) \rightarrow 2 \tau^{1 / 2} z$, we can write
\begin{align*}
I(t, x)&=\int_0^t \int_0^{\infty} \frac{e^{-\frac{|x-y|^2}{4(t-n)}+\left(1-v_0^2\right)(t-s)}}{(4 \pi(t-s))^{1 / 2}} f(s,y) e^{i \omega s} d y d s\\
&=e^{i \omega t} \int_0^t e^{-\left(i \omega+\left(v_0^2-1\right)\right) \tau} \frac{1}{\sqrt{\pi}} \int_{\frac{x}{2\tau^{1/2}}}^{\infty} e^{-z^2} f\left(\tau, x-2 \tau^{1 / 2} z\right) d z d \tau .
\end{align*}

Letting $\varphi(t,x)=\frac{1}{\sqrt{\pi}} \int_{\frac{x}{2\tau^{1/2}}}^{\infty} e^{-z^2} f\left(t, x-2 t^{1 / 2} z\right) d z$ we can verify that
$$
\begin{aligned}
\|\varphi(t,\cdot)\|_{L_x^{\infty}} & \leq\|f(\cdot, t)\|_{L_x^{\infty}} \\
\left\|\partial_t \varphi(t,\cdot)\right\|_{L_x^{\infty}} & =\left\|\frac{1}{\sqrt{\pi}} \int_{\frac{x}{2\tau^{1/2}}}^{\infty} e^{-z^2}\left(\partial_x f\left(t,x-2 t^{1 / 2} z\right)\left(-t^{-1 / 2} z\right)+\partial_t f\left(t,x-2 t^{1 / 2} z\right)\right) d z\right\|_{L_x^{\infty}} \\
& \leq \frac{1}{\sqrt{t \pi}}\left\|\partial_x f(t,\cdot)\right\|_{L_x^{\infty}}+\left\|\partial_t f(t, \cdot)\right\|_{L_x^{\infty}} .
\end{aligned}
$$

For each $x$ we integrate by parts in $\tau$, and use that $v_0^2-1 >0$, to get that for any $(t,x)$
$$
\begin{aligned}
|I(t,x)| & =\left|\int_0^t e^{-\left(i \omega+\left(v_0^2-1\right)\right) \tau} \varphi(t,x) d \tau\right| \\
& \left.=\frac{1}{\left|i \omega+v_0^2-1\right|}\left|e^{-\left(i \omega+\left(v_0^2-1\right)\right) \tau} \varphi(\tau ,x)\right|_{\tau=0}^t-\int_0^t e^{-\left(i \omega+\left(v_0^2-1\right)\right) \tau} \partial_t \varphi(\tau ,x) d \tau \right\rvert\, \\
& \leq \frac{1}{\omega}\left(2\|f\|_Y+\left(\left\|\partial_x f\right\|_Y+\left\|\partial_t f\right\|_Y\right) \int_0^t e^{-\left(v_0^2-1\right) \tau}\left(\frac{1}{\sqrt{\tau \pi}}+1\right) d \tau\right) \\
& \leq \frac{C}{\omega}\left(\|f\|_Y+\left\|\partial_x f\right\|_Y+\left\|\partial_t f\right\|_Y\right),
\end{aligned}
$$
for some $C$ that only depends on $v_0^2-1$. Similarly we find a bound using $\tilde{G_1}(x+y,t)$, and we conclude that 
\begin{equation}\label{integral-estimate-for-w}
    \left\|\Int_0^t e^{\left(1 - v_0^2\right)(t-\tau)}G(t-\tau) \ast f\right\|_Y \leq \frac{C_{v_0}}{\omega}\left(\|f\|_Y+\left\|\partial_x f\right\|_Y+\left\|\partial_t f\right\|_Y\right).
\end{equation}

The idea of the proof of Lemma \ref{lemma-bound-high-frquencyterm} is to consider an iterative approximation of \eqref{highfrequencyequation}. Multiplying the equation by $e^{-\left(1 -v_0^2\right)t}$ we get
$$
\partial_t e^{-\left(1-v_0^2\right) t} w-\partial_x^2 e^{-\left(1-v_0^2\right) t} w=\left(2v_0 h\psi - h^2\psi^2 \right)e^{-\left(1-v_0^2\right) t} w +e^{-\left(1-v_0^2\right) t} \dot{h}\psi .
$$

By virtue of Duhamel's principle, we get
$$
w(t, x)=\Int_0^t e^{\left(1-v_0^2\right)(t-\tau)} G(t-\tau) *\left(\left(2v_0 h\psi - h^2\psi^2 \right)w+\dot{h} \psi \right) d \tau .
$$
Let us consider the following iterative procedure. Set $w^{(0)}=0$ and define 
$$
w^{(k+1)}(t, x)=\Int_0^t e^{\left(1-v_0^2\right)(t-\tau)} G(t-\tau) *\left(\left(2v_0 h\psi - h^2\psi^2 \right)w^{(k)}+\dot{h} \psi \right) d \tau .
$$
Now, the next step is to look at the convergence of the sequence $\{w^{(k)}\}_{k\in\N}$ in the space $C\left([0,T];BC^0(R^+)\right)$. We consider the difference between two consecutive terms and using that $v_0^2 - 1>0$ we obtain
\begin{align*}
    \left\| w^{(k+1)} - w^{(k)} \right\|_Y &\leq \underset{0<t<T}{\sup}\Int_0^t \dfrac{e^{\left( 1 -v_0^2\right)(t-\tau)}}{\sqrt{\pi(t-\tau)}}\left\| 2v_0 h\psi - h^2\psi^2 \right\|_{L^{\infty}_x}\left\| w^{(k)} - w^{(k-1)}  \right\|_Y d\tau\\
    &\leq \left\| 2v_0 h\psi - h^2\psi^2\right\|_{L^{\infty}_x}\left\| w^{(k)} - w^{(k-1)}  \right\|_Y \underset{0<t<T}{\sup}\Int_0^t \dfrac{e^{\left( 1 -v_0^2\right)(t-\tau)}}{\sqrt{\pi(t-\tau)}} d\tau\\
    &\leq \frac{M^2+2 M|v_0|}{\sqrt{v_0^2-1}}\left\| w^{(k)} - w^{(k-1)}\right\|_Y.
\end{align*}
We see that we got the part of the condition \eqref{condition-for-contraction}, so defining as we did before 
$$
\tilde{\alpha}(T) = \frac{M^2+2 M|v_0|}{\sqrt{v_0^2-1}}
$$
and for $m>n$ we see that
$$
\left\|w^{(m)}-w^{(n)}\right\|_Y \leq \sum_{k=n}^{m-1}\left\|w^{(k+1)}-w^{(k)}\right\|_Y \leq \sum_{k=n}^{m-1} \tilde{\alpha}^k\left\|w^{(1)}-w^{(0)}\right\|_Y.
$$
Hence, we have a contraction mapping, therefore the sequence $\{w^{(k)}\}_{k\in \N}$ converges in $C([0,T];BC^0(R^+) \times BC^0(R^+))$. With this in mind, the estimate on $w$ will be obtained by bounding $\left\| w^{(1)} - w^{(0)} \right\|_Y$, since $w^{(0)} = 0$ 
\begin{equation*}
    w^{(1)} - w^{(0)} = \Int_0^t e^{\left( 1 - v_0^2\right)} G (t-\tau) \ast A\omega \cos(\omega \tau)\psi d\tau,
\end{equation*}
so, using \eqref{integral-estimate-for-w} with $f(t,x) = A\omega \psi(x)$, we obtain
\begin{equation*}
    \left\| w^{(1)} - w^{(0)} \right\|_Y \leq AC_1,
\end{equation*}
where $C_1 = C(v_0,\norm{\psi}_{\infty},\norm{\psi^{\prime}}_{\infty})$.
 Finally, using the convergence of the sequence ${w^{(k)}}$ and that $\alpha <1$ we conclude
$$
\begin{aligned}
\left\|w-w^{(0)}\right\|_Y & \leq\left\|w-w^{(N+1)}\right\|_Y+\sum_{k=0}^N\left\|w^{(k+1)}-w^{(k)}\right\|_Y \\
& \leq\left\|w-w^{(N+1)}\right\|_Y+\sum_{k=0}^N \tilde{\alpha}^k\left\|w^{(1)}-w^{(0)}\right\|_Y \\
& \leq\left\|w-w^{(N+1)}\right\|_Y+\frac{AC_1\tilde{\alpha}}{1-\tilde{\alpha}}  .
\end{aligned}
$$
Therefore, by taking the limit as $N\to \infty$ we get \eqref{bound-for-high-frequency-term}. Concluding that,
$$ 
|w(t,x)| \leq C_1 A \quad \text{for all} \quad x\in [0,\infty), 0\leq t \leq T.
$$
Similarly we can get the estimate on $w_{avg}$, i.e., we apply the same iterative process to 
$$
w_{avg}^{(k+1)}(t, x)=\Int_0^t e^{\left(1-v_0^2\right)(t-\tau)} G(t-\tau) *\left(\dfrac{A^2}{2} w_{avg}^{(k)}+ \frac{CA}{\omega} \right) d \tau ,
$$
to get the existence of the solution $w_{avg}$ in $C([0,\infty];\mathcal{B})$. And finally we see that 
\begin{equation*}
    w_{avg}^{(1)} - w_{avg}^{(0)} = \Int_0^t e^{\left( 1 - v_0^2\right)} G (t-\tau) \ast  \frac{CA}{\omega} d\tau,
\end{equation*}
which help us to conclude $|w_{avg}| \leq \dfrac{C_2A}{\omega}$ for all $x\in [0,\infty), 0\leq t \leq T.$ Concluding in this way the proof of Lemma \ref{lemma-bound-high-frquencyterm}.

\section*{Conflict of Interest Statement}
The authors declare that they have no conflict of interest.

\section*{Data Availability Statement}
This study does not involve any datasets. All results are derived analytically and are fully presented within the manuscript.

\bibliographystyle{acm}
\bibliography{bibliography}

\end{document}